\documentclass[12pt]{amsart}
\usepackage{wrapfig}

  \usepackage{appendix}
\usepackage{amssymb}   
\usepackage{amsmath} 
\usepackage[hidelinks,bookmarksdepth=2]{hyperref}
\usepackage{enumitem}
\usepackage{xcolor}
	\definecolor{green}{HTML}{007F00} 
\usepackage{tikz}
	\usetikzlibrary{calc}
	\usetikzlibrary{cd}
	\usetikzlibrary{patterns}
\tikzset{
full/.style={inner sep=2.5pt,draw,fill=black,shape=circle,outer sep=5pt},
half/.style={inner sep=2.5pt,draw,pattern=crosshatch dots,pattern color=black,
shape=circle,outer sep=5pt},
empty/.style={inner sep=2.5pt,draw,fill=white,shape=circle,outer sep=5pt}
}
\usepackage[normalem]{ulem}
\usepackage{mathtools}
\usepackage{etoolbox}
\usepackage{varioref}
\hypersetup{
    colorlinks,
    linkcolor={red!50!black},
    citecolor={blue!50!black},
    urlcolor={blue!80!black}
}
\usepackage{tikz,pgfplots}
 \pgfplotsset{compat=1.18} 
 \usetikzlibrary{arrows}

\usepackage{graphicx}
\newcommand*{\Scale}[2][4]{\scalebox{#1}{\ensuremath{#2}}}%

\usepackage{multirow}

\usepackage{comment}
\input{insbox}
\usepackage{hyperref}
\usepackage{cleveref}

\newtheorem{theorem}{Theorem}
\newtheorem{lemma}[theorem]{Lemma}

\newtheorem{definition}[theorem]{Definition}
\newtheorem{corollary}[theorem]{Corollary}

\theoremstyle{remark}
\newtheorem{remark}[theorem]{Remark}
\newtheorem{example}[theorem]{Example}

\newtheorem{exercise}[theorem]{Exercise}

\newcommand{\Z}{\mathbb{Z}}
\newcommand{\Q}{\mathbb{Q}}
\newcommand{\R}{\mathbb{R}}
\newcommand{\C}{\mathbb{C}}

\newcommand{\bbF}{\mathbb{F}}

\newcommand{\bbQ}{\mathbb{Q}}
\newcommand{\bbR}{\mathbb{R}}
\newcommand{\bbZ}{\mathbb{Z}}
\newcommand{\Cc}{\mathbb{C}^{\times}}

\newcommand{\ra}{\rightarrow}

\newcommand{\sra}{\twoheadrightarrow}
\newcommand{\xra}{\xrightarrow}
\newcommand{\SL}{\operatorname{SL}}

\newcommand{\depth}{\operatorname{depth}}

\newcommand{\Aut}{\operatorname{Aut}}

\newcommand{\Irr}{\operatorname{Irr}}
\newcommand{\Ind}{\operatorname{Ind}}
\newcommand{\Res}{\operatorname{Res}}
\newcommand{\Tr}{\operatorname{Tr}}
\newcommand{\Nm}{\operatorname{Nm}}

\newcommand{\Group}{\operatorname{I}}
\newcommand{\Gal}{\operatorname{Gal}}
\newcommand{\val}{\operatorname{val}}

 \newcommand{\Hom}{\operatorname{Hom}}
 \newcommand{\Stab}{\operatorname{Stab}}
 \newcommand{\Gm}{\mathbb{G}_m}
 \DeclareMathOperator{\cl}{cl}
 \DeclareMathOperator{\Norm}{{N}}
 \DeclareMathOperator{\Sh}{{Sh}}
 \DeclareMathOperator{\pr}{{pr}}
 \DeclareMathOperator{\Fr}{{Fr}}
 \DeclareMathOperator{\LLC}{{LLC}}
 \DeclareMathOperator{\ab}{{ab}}
 \DeclareMathOperator{\res}{res}

\newcommand\mathdef\textit
\newcommand\tn\textnormal
\newcommand\defeq{:=}

\begin{document}

\title{Normalized Indexing for Ramification Subgroups}

\author{Stephen DeBacker}
\address{University of Michigan\\
Ann Arbor, MI 48109-1043, USA}
\email{stephendebacker@umich.edu}

\author{David Schwein}
\address{University of Utah,
Salt Lake City, UT 84112, USA}
\email{david.schwein@utah.edu}

\author{Cheng-Chiang Tsai}
\address{Institute of Mathematics, Academia Sinica, 6F, Astronomy-Mathematics Building, No. 1,
Sec. 4, Roosevelt Road, Taipei, Taiwan \vskip.2cm
also Department of Applied Mathematics, National Sun Yat-Sen University, and Department of Mathematics, National Taiwan University}

\email{chchtsai@gate.sinica.edu.tw}

\begin{abstract} This expository note introduces a normalization of the indexing of the lower and upper numbering ramification subgroups of local class field theory.  We then look at how this  normalization interacts with base change for Langlands parameters.
\end{abstract}

 \makeatletter
\let\@wraptoccontribs\wraptoccontribs
\makeatother

  \subjclass[2020]{11S15, 22E50}


  \date{\today}

\maketitle

\setcounter{tocdepth}{1}
\tableofcontents

\section*{Introduction}
\noindent Suppose $F$ is a nonarchimedean local field, and $\bar{F}$ is a fixed separable closure of $F$.
In the standard treatment of ramification  groups, the valuation is chosen so that the lower numbering subgroups are indexed by integers.  In this expository note we fix the unique valuation $\val: \bar{F} \rightarrow \Q \cup \{ \infty \}$ for which $\val(F^\times) = \Z$, and we  use $\val$ to define a new indexing of the lower and upper numbering ramification groups.  This new  indexing seems to be better suited to base change.

\subsection*{Notation}

The residue field, $k_F$, of $F$ has  characteristic $p$.

We will assume everywhere, except in \Cref{sec:nongalois},
that finite separable extensions of $F$ are contained in~$\bar{F}$.
 For any finite separable extension $E/F$ we denote by $R_E$ the ring of integers of $E$ and by $W_E$ the Weil group of $E$.

Suppose $F \leq E \leq L$ is a tower of (finite) separable extensions with $L/E$ Galois.  We let $\Group(L/E)$ denote  the inertia subgroup of $\Gal(L/E)$.  Thus, $L^{\Group(L/E)}$ is the maximal unramified subextension of $L/E$.

\section{Ramification groups}

  \subsection{Lower numbering}

  \begin{definition}
     Suppose $L/E$ is finite Galois and $\varpi_L$ is a uniformizer for $L$.  \begin{itemize}
         \item 
   For $\sigma \in \Group(L/E)$
     set
\[\depth_{L/E}(\sigma) := \val \left( \dfrac{ \sigma(\varpi_L) - \varpi_L}{\varpi_L} \right).\]
  \item For $r \in \R_{\geq 0}$ we set
 \[\Group(L/E)_r := \{ \sigma \in \Group(L/E) \, : \, \depth_{L/E}(\sigma) \geq r \}. \]
  \end{itemize}
  \end{definition}

\begin{remark}
  Since $\sigma$ lies in the inertia subgroup, the definition of $\depth_{L/E}$ is independent of the choice of uniformizer. 
  Indeed, any other uniformizer is of the form $c\varpi_L$ for $c\in R_L^\times$, and the claim follows from the identity
  \[
  \sigma(c\varpi_L) - c\varpi_L = c(\sigma(\varpi_L)-\varpi_L) + \sigma(\varpi_L)(\sigma(c)-c).
  \]
  One often makes similar sorts of arguments in proving basic properties of the depth function.
\end{remark}

The definition of the ramification groups in lower numbering goes back at least to Hilbert \cite[\S44]{Hil97}.

  \begin{lemma} \label{depth-properties}
      The function $\depth_{L/E}$ satisfies the following properties:
      \begin{enumerate}
          \item $\depth_{L/E}(\sigma^{-1}) = \depth_{L/E}(\sigma)$.
          \item $\depth_{L/E}(\sigma\tau)\geq\min\bigl(\depth_{L/E}(\sigma),\depth_{L/E}(\tau)\bigr)$, with equality if $\sigma$ and $\tau$ have different depths.
          \item If $K/E$ is finite Galois with  $E \leq K \leq L$, then $\depth_{L/E}(\tau) = \depth_{L/K}(\tau)$ for all $\tau \in \Group(L/K)$. 
      \end{enumerate}
  \end{lemma}

  \begin{proof}
      The first part follows from the fact that automorphisms preserve valuations, the second part follows from the identity
      \[
      \sigma\tau(\varpi_L) - \varpi_L = \sigma(\tau(\varpi_L) - \varpi_L) + (\sigma(\varpi_L) - \varpi_L),
      \]
      and the third part is clear.
  \end{proof}

We define
\[
\Group(L/E)_{r+} := \bigcup_{s > r} \Group(L/E)_s
\]
and say that $t$ is a \emph{jump} in the (lower numbering) filtration if $\Group(L/E)_t\neq\Group(L/E)_{t+}$.

  \begin{lemma} \label{lem:3} Suppose $L/E$ is finite Galois.
  \begin{enumerate}
  \item  If $t$ is a jump in the lower numbering filtration then $t\in\frac{1}{e(L/F)}\cdot\Z$.
  \item   $\Group(L/E)_0 = \Group(L/E)$ and $\Group(L/E)_{r} = \Group(L/E)_{1/e(L/F)}$ for $0 < r \leq 1/e(L/F)$.
\item $\bigcap_{r \geq 0} \Group(L/E)_r$ is trivial.
\item If $K/E$ is finite Galois with  $E \leq K \leq L$, then 
$\Group(L/E)_r  \cap \Group(L/K)= \Group(L/K)_r$.
\item If $t \geq s \geq 0$, then $\Group(L/E)_t \trianglelefteq \Group(L/E)_s$.
\end{enumerate}
  \end{lemma}

  \begin{proof}
      The first four parts follow directly from the definitions while the last part follows from \Cref{depth-properties}.
  \end{proof}

Given $t \geq s \geq 0$, we define the group
\[
\Group(L/E)_{s:t} := \Group(L/E)_s/\Group(L/E)_t
\] 
        We denote by $\ell(L/E)$ the largest (that is, deepest) jump in the filtration.  That is, $\ell(L/E) = \inf \{ r \in \R_{\geq 0} \colon \text{$\Group(L/E)_r$ is trivial}\}$.
        The $\ell$ here is short for ``lower numbering''; later, we define a similar quantity $u$ for the ``upper numbering''.

\begin{lemma} \label{lem:lemma5} Suppose $r \in \R_{\geq 0}$.   The map $\Group(L/E)_r \rightarrow (L^\times)_r$ that sends $\sigma$ to $\sigma(\varpi_L)/\varpi_L$ induces an injective group homomorphism $\theta_r \colon \Group(L/E)_{r:r+} \hookrightarrow (L^\times)_{r:r+}$.  Additionally,
\begin{enumerate}
    \item  $\theta_r$ is independent of the choice of uniformizer $\varpi_L$ for $L$.
    \item If $\sigma \in \Group(L/E)_{0}$ and $\tau \in \Group(L/E)_{r} \setminus \Group(L/E)_{r+}$ with $r > 0$, then
    \[ \theta_r(\sigma \tau \sigma^{-1}) = \theta_0(\sigma)^{\lceil r \cdot e(L/F) \rceil} \theta_r(\tau).\]
    (Here we are thinking of $\theta_r$ as a group homomorphism $\Group(L/E)_r \rightarrow k_L$ in the natural way.)
    \item If $s,t > 0$, then $[\Group(L/E)_t,\Group(L/E)_s] \leq \Group(L/E)_{(t+s)+}$
    \item If $\Group(L/E)_{t:t+}$ and $\Group(L/E)_{s:s+}$ are both nontrivial, then $p$ divides $e(L/F) \cdot (t-s)$.
\end{enumerate}
\end{lemma}

\begin{proof}
    This is proved in~\cite[IV.2]{serre:local}.  Specifically, we are using Propositions~6, 7, 9, 10, and 11 of \emph{loc. cit.}   To translate, note that
the group $G_i$ of \emph{loc. cit.} is $\Group(L/K)_{i/e(L/F)}$ in this note.  
    \end{proof}

\begin{remark}  Since $[\Group(L/E)_r,\Group(L/E)_r] \leq \Group(L/E)_{r+}$ for all $r \geq 0$, we see that $\Group(L/E)$ is solvable.   We also see that $\Group(L/E)_{0:0^+}$ is a cyclic group of order prime to $p$ and each $\Group(L/E)_{r:r^+}$ is a direct sum of cyclic groups of order $p$. 
\end{remark}

\begin{remark}   As follows from, for example,~\cite[IV.1 Lemma 1]{serre:local}, if $\alpha_{L/E}$ is a generator for $R_L$ over $R_E$, then 
\[\depth_{L/E}(\sigma) = \val \left(\varpi_L^{-1} \cdot (\sigma(\alpha_{L/E}) - \alpha_{L/E} )\right).\]
\end{remark}

\begin{exercise}
    If $r>0$, show that the map $\theta_r$ of Lemma~\ref{lem:lemma5} can be extended to an injective group homomorphism $\Group(L/E)_{r:2r} \hookrightarrow (L^\times)_{r:2r}$.  Show that, in general, this map depends on the choice of $\varpi_L$.
\end{exercise}

\subsection{Upper numbering}
\hspace{-.6em}\footnote{The treatment  of the material in this section was inspired by Wotjek Wawr\'ow's note \emph{Higher Ramification Groups}.}
Suppose $L/E$ is finite Galois and $E \leq K \leq L$ with $K/E$ Galois.  We now want to look at $\Group(K/E) = \Group(L/E)/\Group(L/K)$.  
  For any $\sigma \in \Group(L/E)$ we define $\bar{\sigma} = \sigma \Group(L/K) \in \Group(K/E)$.
  
 \begin{lemma} \label{lem:firstformula} If $\sigma \in \Group(L/E)$, then
\begin{equation} \tag{$\clubsuit$}  \label{descent:formula}
\depth_{K/E} (\bar{\sigma}) = \sum_{\tau \in \Group(L/K)} \depth_{L/E}(\sigma \tau).
\end{equation}
\end{lemma}

     \begin{proof}       
     First, assume that $L/E$ is totally ramified. If $\bar{\sigma}$ is the identity in $\Group(K/E)$, then both sides of (\ref{descent:formula}) are infinity. Suppose now that $\bar{\sigma}$ is not the identity. Then \(\depth_{K/E} (\bar{\sigma})\) is given by
          \[  
	  \val \left(\bar{\sigma}(\varpi_K) - \varpi_K \right)-\val(\varpi_K)=  \val \left( \sigma(\varpi_K) - \varpi_K  \right)-\val(\varpi_K).
	  \]
Since
         \[
	 \val(\varpi_K)=|\Group(L/K)|\cdot\val(\varpi_L)
         \]
we need to show that 
\[
\sigma(\varpi_K) - \varpi_K  \text{\,  and } \prod_{\tau \in \Group(L/K)} (\sigma \tau (\varpi_L) - \varpi_L)  
\]
have the same valuation, i.e., they generate the same ideal in $R_L$. For this, we may use Tate's proof of this fact, which applies because $\Gal(E/F) = \Group(E/F)$.  Tate's proof may be found in the proof of Proposition 3 on page 63 of Serre's \emph{Local Fields}~\cite{serre:local}.

\InsertBoxR{0}{%
\begin{tikzcd}[arrows={dash},column sep=0.2em,row sep=tiny,shorten=-1mm,
ampersand replacement=\&]
L \arrow[dd,shorten=0mm] \drar{\tn{ram}} \& \\
\& K' \dlar{\tn{unr}} \drar{\tn{ram}} \& \\
K \arrow[dd,shorten=0mm] \drar{\tn{ram}} \&\& E'' \dlar{\tn{unr}} \\
\& E' \dlar{\tn{unr}} \\
E
\end{tikzcd}
}[4]
When $L/E$ is not  totally ramified,
we reduce to the case where $L/E$ is totally ramified as follows.
Let $E'$ be the maximal unramified subextension of $K/E$,
 let $K'$ be the maximal unramified subextension of $L/K$, and let $E''$ be the maximal unramified subextension of $K'/E'$,
which is also the maximal unramified subextension of~$L/E$.
We have displayed the resulting tower of extensions on the right,
labeling the unramified and totally ramified inclusions
with ``unr'' and ``ram'', respectively.
Then $\Group(K/E) = \Group(K/E')$ while restriction to $K$
induces an isomorphism $\Group(K'/E'')\simeq\Group(K/E')$.
Using this identification, together with the fact that
any uniformizer for $K$ is also a uniformizer for~$K'$,
we see that
\[
\depth_{K/E}(\bar\sigma) = \depth_{K/E'}(\bar\sigma)
= \depth_{K'/E''}(\bar\sigma).
\]
Moreover, 
\[
\sum_{\tau\in\Group(L/K)} \depth_{L/K}(\sigma\tau)
= \sum_{\tau\in\Group(L/K')} \depth_{L/K'}(\sigma\tau)
\]
because $\Group(L/K) = \Group(L/K')$.
So we may replace $L/K/E$ by $L/K'/E''$,
reducing to the previous case.
\end{proof}

\begin{definition} \label{relative-depth}
 For $\sigma \in \Group(L/E)$ we let $\depth_{L/E}^{L/K}(\sigma)$ denote the maximal depth attained by an element of the coset $\sigma \Group(L/K)$.  That is,
 \[\depth_{L/E}^{L/K}(\sigma) = \max \{ \depth_{L/E}(\sigma \tau) \, | \, \tau \in \Group(L/K) \}. \]  
\end{definition}

\begin{lemma} \label{lem:six}
 If $\sigma \in \Group(L/E)$, 
 then
\[\depth_{K/E}(\bar{\sigma}) =   \sum_{\tau \in \Group(L/K)} \min( \depth_{L/E}(\tau) , \depth^{L/K}_{L/E} (\sigma)). \]
\end{lemma}

\begin{remark}
    Here and elsewhere it is useful to remember that for all $\tau \in \Group(L/K)$ we have $\depth_{L/E}(\tau)  = \depth_{L/K}(\tau)$.
\end{remark}

\begin{proof}
We may assume without loss of generality that 
 $\depth_{L/E}(\sigma) =  \depth_{L/E}^{L/K}(\sigma)$.
  Suppose $\tau \in \Group(L/K)$. 
Either $\depth_{L/E}(\sigma \tau) =  \depth_{L/E}^{L/K}(\sigma)$ 
or  $\depth_{L/E}(\sigma \tau) < \depth_{L/E}^{L/K}(\sigma)$.  In the latter case $\depth_{L/E}(\sigma \tau) = \depth_{L/E}(\tau)$,  so  
\[ \depth_{L/E}(\sigma \tau) = \min(\depth_{L/E}(\tau), \depth_{L/E}^{L/K}(\sigma)). \]
The result now follows from Lemma~\ref{lem:firstformula}.
\end{proof}

 \begin{lemma} \label{lem:seven}  Suppose $r \geq 0$.  We have $\Group(L/E)_r \Group(L/K)/\Group(L/K) = \Group(K/E)_s$ where 
\[s = \sum_{\tau \in \Group(L/K)} \min( \depth_{L/E}(\tau) , r). \]
\end{lemma}

\begin{proof}   Suppose $\sigma \in \Group(L/E)$.  

We have $\sigma \in \Group(L/E)_r \Group(L/K)$ if and only if $\depth_{L/E}^{L/K} (\sigma) \geq r$.  Since  $\depth_{L/E}(1) = \infty$, the function 
\[ x \mapsto  \sum_{\tau \in \Group(L/K)} \min( \depth_{L/E}(\tau) ,x) \]
is strictly increasing.  Thus, $\depth_{L/E}^{L/K} (\sigma) \geq r$ if and only if 
\begin{equation*}
    \begin{split}
\depth_{K/E}(\bar{\sigma}) &=   \sum_{\tau \in \Group(L/K)} \min( \depth_{L/E}(\tau) , \depth^{L/K}_{L/E} (\sigma))\\ &\geq  \sum_{\tau \in \Group(L/K)} \min( \depth_{L/E}(\tau) , r) = s.
\end{split}
\end{equation*}
Since $\bar{\sigma} \in \Group(K/E)_s$ if and only if 
$\depth_{K/E}(\bar{\sigma}) \geq s$, the lemma is proved.
\end{proof}

\begin{definition}  \label{def:newphi} For $x \in \R_{\geq 0}$ we define a  function
$$\varphi_{L/E}(x) = \sum_{\sigma \in \Group(L/E)} \min( \depth_{L/E}(\sigma),x)$$
\end{definition}

\begin{remark}
The function $\varphi_{L/E} \colon \R_{\geq 0} \rightarrow \R_{\geq 0}$ is piecewise linear, continuous, $0$ at $0$,  strictly increasing, and bijective.  Moreover, at those $x \geq 0$  for which $\varphi_{L/E}$ is differentiable, we have
 \begin{equation*}
 \varphi'_{L/E}(x) = \sum_{\sigma \in \Group(L/E)} \min(\depth_{L/E}(\sigma),x)'  = \sum_{\substack{\sigma \in \Group(L/E) \\ \depth_{L/E}(\sigma) \geq x}} 1 
    = 
    \left| \Group(L/E)_x \right|.
    \end{equation*}  
Note that this function is not convex (like $x \mapsto e^x$), but it  is concave.  
\end{remark}

\begin{definition}
    Define $\psi_{L/E} \colon \R_{\geq 0} \rightarrow \R_{\geq 0}$ by $\psi_{L/E}(x) = \varphi_{L/E}^{-1}(x)$. 
\end{definition}

    \begin{remark}  The function $\psi_{L/E}$ is a normalized version of the Hasse-Herbrand function.
    \end{remark}
    
We can now recast Lemma~\ref{lem:six} and Lemma~\ref{lem:seven}.
\begin{lemma}  \label{lem:11} Recall that $E/F$ is a finite separable extension.  Let $E \leq K \leq L$ be a tower of finite extensions of $E$ with $L/E$ and $K/E$ Galois.
\begin{enumerate}
    \item If $\sigma \in \Group(L/E)$, then
    \[ \depth_{K/E}(\bar{\sigma}) = \varphi_{L/K}(\depth_{L/E}^{L/K}(\sigma)) = \max_{\tau \in \Group(L/K)} \varphi_{L/K}(\depth_{L/E}(\sigma \tau)).\]
    \item If $r \in \R_{\geq 0}$, then
    \[ \Group(L/E)_r \Group(L/K)/\Group(L/K) = \Group(K/E)_{\varphi_{L/K}(r)}.\]
\end{enumerate}
\end{lemma}

\begin{corollary} \label{cor:19}
Suppose $L/K/E$ is a tower of Galois extensions and $s \geq 0$.  The following sequence is exact.
 \pushQED{\qed} 
\[ 1 \ra \Group(L/K)_s \ra \Group(L/E)_s \ra \Group(K/E)_{\varphi_{L/K}(s)} \ra 1. \qedhere \]
    \popQED 
\end{corollary}

\begin{corollary} \label{cor:exact2}
    Suppose $L/K/E$ is a tower of Galois extensions and $s \geq 0$.  We have
\[ s > \ell(L/E) \text{ if and only if } s > \ell(L/K) \text{ and } \varphi_{L/K}(s) > \ell(K/E). \]
\end{corollary}

\begin{proof}
Consider the short exact sequence of \Cref{cor:19}.
The middle group is trivial if and only if both outer groups are trivial.
\end{proof}

\begin{lemma}  \label{lem:12}
If $E \leq K \leq L$ is a tower of fields with $K/E$ and $L/E$ finite Galois, then
\[
\varphi_{L/E} = \varphi_{K/E} \circ \varphi_{L/K},
\qquad
\psi_{L/E} = \psi_{L/K}\circ\psi_{K/E}.
\] 
\end{lemma}

\begin{proof}
The second part follows from the first part by taking the inverse of both sides.
For the first part, since both sides are piecewise linear continuous functions that agree at $0$,
it is enough to show that their derivatives agree at all but finitely many points.

Choose $x \in \R_{\geq 0}$ such that the three functions 
$\varphi_{L/E}$,   $\varphi_{K/E}$, and $\varphi_{L/K}$ are all differentiable at $x$.
Then
\[
\varphi'_{L/E}(x) = |\Group(L/E)_x|,
\]
and reusing this result twice (after replacing $L/E$ with $L/K$ or $K/E$),
we find that
\[
(\varphi_{K/E}\circ\varphi_{L/K})'(x)
= \varphi'_{K/E}(\varphi_{L/K}(x))\cdot\varphi'_{L/K}(x)
= |\Group(K/E)_{\varphi_{L/K}(x)}|\cdot|\Group(L/K)_x|.
\]
Now the proof is completed using \Cref{cor:19}.
\end{proof}

\begin{definition} \label{def:23}
Suppose $L/E$ is a finite Galois extension and $s \geq 0$.  We define 
\[ \Group(L/E)^s := \Group(L/E)_{\psi_{L/E}(s)}. \]
\end{definition}

\begin{lemma}   \label{lem:upperquot}
If $s \in \R_{\geq 0}$ and  $E \leq K \leq L$ is a tower of fields with $K/E$ and $L/E$ finite Galois, then
\[ \Group(K/E)^s = \Group(L/E)^s \Group(L/K)/ \Group(L/K).\] 
\end{lemma}

\begin{proof}
By definition, $\Group(L/E)^s \Group(L/K)/ \Group(L/K) 
= \Group(L/E)_{\psi_{L/E}(s) }\Group(L/K)/ \Group(L/K)$.
By \Cref{cor:19}, this quotient is $\Group(K/E)_{\varphi_{L/K}(\psi_{L/E}(s))}$.
Now \Cref{lem:12} implies that
\[
\varphi_{L/K}\circ\psi_{L/E}
= \varphi_{L/K}\circ\psi_{L/K}\circ\psi_{K/E} = \psi_{K/E},
\]
and the proof is finished by again applying \Cref{def:23}.
\end{proof}

\begin{exercise} \label{exc:alternateformsofexact}  Suppose $L/K/E$ is a tower of Galois extensions.  Show that  for all $t \geq 0$ the following sequences are exact.
\[ 1 \ra \Group(L/K)_{\psi_{L/E}(t)} \ra \Group(L/E)^t \ra \Group(K/E)^t \ra 1.  \]
\[ 1 \ra \Group(L/K)^{\psi_{K/E}(t)} \ra \Group(L/E)^t \ra \Group(K/E)_{\psi_{K/E}(t)} \ra 1.  \]
\[ 1 \ra \Group(L/K)^{\psi_{K/E}(t)} \ra \Group(L/E)^t \ra \Group(K/E)^t \ra 1.  \]
\[ 1 \ra \Group(L/K)^{t} \ra \Group(L/E)_{\psi_{L/K}(t)} \ra \Group(K/E)_{t} \ra 1.  \]
\end{exercise}

\subsection{Some examples}
\subsubsection{Example: Unramified and tamely ramified extensions}  \label{ex:unram} Suppose $L/E/F$ is a tower of fields as usual.   Note that $E \leq L^{\Group(L/E)_0} \leq L^{\Group(L/E)_{0^+} } \leq L$, and by construction we have 
\begin{itemize}
    \item $L/E$ is unramified if and only if $L = L^{\Group(L/E)_0}$ and
    \item $L/E$ is tamely ramified if and only if $L = L^{\Group(L/E)_{0+}}$.
\end{itemize} 
Thus, if $L$ is an unramified or tamely ramified extension of $E$, then 
$\Group(L/E)_{0+}$ is trivial and so $\varphi_{L/E} \colon \R_{\geq 0} \rightarrow \R_{\geq 0}$ is the identity map.  It follows that in these situations we have $\Group(L/E)^r = \Group(L/E)_r$ for all $r \geq 0$.   Graphs of $\varphi_{L/E}$ may be found in Figures~\ref{fig:unram} and~\ref{fig:tame}.
\begin{figure}[ht]
    \centering
    \begin{minipage}{0.5\textwidth}
        \centering
      \begin{tikzpicture}
 \draw (-1.2,0) -- (4,0);
 \draw (0,-1.2) -- (0,4);
 \draw[color = green, thick] (0.0,0.0) -- (4.2,4.2);
 
\draw (.05,1) -- (-.05,1);
\draw (.05,2) -- (-.05,2);
\draw (.05,3) -- (-.05,3);

\draw (1,-.05) -- (1,.05);
\draw (2,-.05) -- (2,.05);
\draw (3,-.05) -- (3,.05);

  \draw (-.2,1) node { \Scale[.75]{1}};
 \draw (-.2,2) node { \Scale[.75]{2}};
 \draw (1,-.2) node { \Scale[.75]{1}};
  \draw (2,-.2) node { \Scale[.75]{2}};

  \draw (.2,4) node { \Scale[1]{s}};
     \draw (4,.2) node { \Scale[1]{r}};

   \draw (-1.4,4) node { \Scale[1]{\Group(L/E)^s}};
     \draw (4,-0.9) node { \Scale[1]{\Group(L/E)_r}};
 
 \filldraw[color = green] (0,0) circle (2pt);

 \draw[color = green] (-.8,2.5) node{ $1$};

 \draw[color = green] (-.5,0) -- (-.5, 4.5);
  \filldraw[color = green] (-.5,0) circle (2pt);
 \draw[color = green] (2.5,-.8) node{ $1$};

   \draw[color = green] (0,-.5) -- (4.5,-.5);
  \filldraw[color = green] (0,-.5) circle (2pt);
                \end{tikzpicture}

        \caption{$L/E$ unramified}  \label{fig:unram}
    \end{minipage}\hfill
    \begin{minipage}{0.5\textwidth}
        \centering
       \begin{tikzpicture}
 \draw (-1.2,0) -- (4,0);
 \draw (0,-1.2) -- (0,4);
 \draw[color = green, thick] (0.0,0.0) -- (4.2,4.2);
 
\draw (.05,1) -- (-.05,1);
\draw (.05,2) -- (-.05,2);
\draw (.05,3) -- (-.05,3);

\draw (1,-.05) -- (1,.05);
\draw (2,-.05) -- (2,.05);
\draw (3,-.05) -- (3,.05);

  \draw (-.2,1) node { \Scale[.75]{1}};
 \draw (-.2,2) node { \Scale[.75]{2}};
 \draw (1,-.2) node { \Scale[.75]{1}};
  \draw (2,-.2) node { \Scale[.75]{2}};

  \draw (.2,4) node { \Scale[1]{s}};
     \draw (4,.2) node { \Scale[1]{r}};

   \draw (-1.4,4) node { \Scale[1]{\Group(L/E)^s}};
     \draw (4,-0.9) node { \Scale[1]{\Group(L/E)_r}};
 
 \filldraw[color = blue] (0,0) circle (2pt);
 
\draw[color = blue] (-.6,-.4) node{ \Scale[.6]{ \Group(L/E)}};
 \draw[color = green] (-.8,2.5) node{ $1$};

 \draw[color = green] (-.5,0) -- (-.5, 4.5);
  \filldraw[color = blue] (-.5,0) circle (2pt);
 \draw[color = green] (2.5,-.8) node{ $1$};

 \draw[color = green] (0,-.5) -- (4.5,-.5);
  \filldraw[color = blue] (0,-.5) circle (2pt);

                \end{tikzpicture}

        \caption{$L/E$ tamely ramified, but not unramified} \label{fig:tame}
    \end{minipage}
\end{figure}

\begin{wrapfigure}[17]{R}{0.4\linewidth}
\centering
\vspace{2.5em}
  \begin{tikzpicture}
 \draw (-1.2,0) -- (4,0);
 \draw (0,-1.2) -- (0,4);
  \draw[color = blue, thick] (0.0,0.0) -- (1,2);
 \draw[color = green, thick] (1,2) -- (3.2,4.2);
 
\draw (.05,1) -- (-.05,1);
\draw (.05,2) -- (-.05,2);

\draw (1,-.05) -- (1,.05);

  \draw (-.24,2) node { \Scale[.75]{2 \ell}};
 \draw (1,-.2) node { \Scale[.75]{\ell}};

  \draw (.2,4) node { \Scale[1]{s}};
     \draw (4,.2) node { \Scale[1]{r}};

   \draw (-1.4,4) node { \Scale[1]{\Group(L/E)^s}};
     \draw (4,-0.9) node { \Scale[1]{\Group(L/E)_r}};

 \filldraw[color = blue] (0,0) circle (2pt);
  \filldraw[color = blue] (-.5,0) circle (2pt);
 \filldraw[color = blue] (1,2) circle (2pt);

 \draw[color = blue] (-1.0,1.0) node{\Scale[.6] {\Z/2\Z}};
  \draw[color = green] (-.8,3.0) node{\Scale[.6] {1 }};

 \draw[color = blue] (-.5,0) -- (-.5, 2.0);
  \filldraw[color = blue] (-.5,2) circle (2pt);
 
  \filldraw[color = blue] (-.5,2) circle (2pt);
 \draw[color = green] (-.5,2.0) -- (-.5, 4.5);

 \draw[color = blue] (0,-.5) -- (1.0,-.5);
  \filldraw[color = blue] (0,-.5) circle (2pt);
  \filldraw[color = blue] (1,-.5) circle (2pt);
   \draw[color = green] (1,-.5) -- (4.5,-.5);
\draw[color = blue] (.5,-.8) node{\Scale[.6] {\Z/2\Z}};
  \draw[color = green] (2.5,-.8) node{ \Scale[.6]{1}};

                \end{tikzpicture}

\end{wrapfigure}
\subsubsection{Example: wildly ramified separable quadratic extensions} \label{ex:quadratic}  Suppose the characteristic of the residue field of $F$ is two, and let $L$ be a separable quadratic extension of $E = F$.
Write  $\Gal(L/E) = \langle \sigma \rangle$, and  fix a uniformizer $\varpi_L \in L$.  Let $x^2 + ax + b$ be the minimal polynomial for $\varpi_L$ over $E$.
Set $\ell = \ell(L/E) = \depth_{L/E}(\sigma)$.  According to~\cite[Lemma 41.1]{BH:GLtwo} we have 
\[\displaywidth=\parshapelength\numexpr\prevgraf+2\relax
2 \ell = \min \{ \val(4), 2\val(a) - 1 \}.\]
So $2 \ell$ is odd unless $F$ has characteristic zero and $2 \ell = \val(4)$.
Since $\Group(L/E)$ has order two, we conclude that 
\begin{equation*}
\displaywidth=\parshapelength\numexpr\prevgraf+2\relax
\varphi_{L/E}(r) =
    \begin{cases}
        2r & \qquad 0 \leq r \leq \ell \\
        \ell + r & \qquad \ell < r.
    \end{cases}
\end{equation*}
A graph of $\varphi_{L/E}$ is at right.

\subsubsection{Example: quaternionic Galois groups}   Let $Q = \{\pm 1, \pm i, \pm j, \pm k \}$ be the quaternion group with the usual relations.  

Our first example of a quaternionic extension is taken from~\cite[IV.3 Exc.~2]{serre:local}.  Serre shows~\cite[\S4]{serre:sur} that there exist fields $F$, $E$, and $L$ such that
\begin{enumerate}
    \item $L$ is a totally ramified Galois extension of $E$,
    \item $\Gal(L/E) = Q$, and
    \item $\Group(L/E)_{4/e(L/F)}$ is trivial.
\end{enumerate}
To ease notation, we assume $F = E$.
Since $L/E$ is totally ramified, we have $\Group(L/E) = \Gal(L/E)$.  Thus, $e(L/F) = e(L/E) = 8$.

Since $\Group(L/E)_{0:0^+}$ injects into $k_L$, it is a cyclic group. 
As $Q$ is not cyclic, we must therefore have that $\Group(L/F)_{0+}$ is nontrivial.  Since every subgroup and quotient group of $Q$ is a two group, we conclude that the characteristic of $k_F = k_E$ is two.  This forces $\Group(L/F)_{0:0^+} \leq k_F^\times$ to have an odd number of elements, hence it must be trivial.
Since the commutator subgroup of $Q$ is $Z=\{ \pm 1 \}$, there are two more jumps.   The next jump can't be at $2/8$ or $3/8$: if it were,  then we must have $i,j \in \Group(L/E)_{2/8}$ which implies $-1 \in \Group(L/E)_{4/8+} \leq \Group(L/E)_{4/8} = 1$, a contradiction.  Therefore, the next jump must happen at $1/8$.   If the other jump happens at $t/8$, then by Lemma~\ref{lem:lemma5} we have that $2$ must divide $8 \cdot (t/8 - 1/8)$.  Thus, the final jump happens at $3/8$. We therefore have $\Group(L/E)_r$ is $Q$ for $0 \leq r  \leq 1/8$, it is  $Z$ for $1/8 < r  \leq 3/8$, and it is  $1$ for $3/8 < r$.  Consequently, $\Group(L/E)^s$ is $Q$ for $0 \leq s  \leq 1$, it is  $Z$ for $1 < s  \leq 3/2$, and it is  $1$ for $3/2 < s$.

We graph $\varphi_{L/E}$ below.

 \begin{tikzpicture}
 \draw (-1.2,0) -- (10,0);
 \draw (0,-1.2) -- (0,10);
 \draw[color = blue, thick] (0.0,0.0) -- (.5,4.0);
  \draw[color = red, thick] (.5,4.0) -- (1.5,6);
   \draw[color = green, thick]  (1.5,6) -- (6,10);
\draw (.05,2) -- (-.05,2);
\draw (.05,4) -- (-.05,4);
\draw (.05,6) -- (-.05,6);
\draw (.05,8) -- (-.05,8);

\draw (2,-.05) -- (2,.05);
\draw (4,-.05) -- (4,.05);
\draw (6,-.05) -- (6,.05);
\draw (8,-.05) -- (8,.05);

  \draw (-.2,4) node { \Scale[.75]{1}};
 \draw (-.2,8) node { \Scale[.75]{2}};
 \draw (4,-.2) node { \Scale[.75]{1}};
  \draw (8,-.2) node { \Scale[.75]{2}};

  \draw (.2,10) node { \Scale[1]{s}};
     \draw (10,.2) node { \Scale[1]{r}};

   \draw (-1.4,10) node { \Scale[1]{\Group(L/E)^s}};
     \draw (10,-0.9) node { \Scale[1]{\Group(L/E)_r}};
 
 \filldraw[color = blue] (0,0) circle (2pt);
  \filldraw[color = blue] (.5,4) circle (2pt);
   \filldraw[color = red] (1.5,6) circle (2pt);

 \draw[color = blue] (-.8,2) node{ $Q$};
 \draw[color = red] (-.8,5) node{ $Z$};
 \draw[color = green] (-.8,8) node{ $1$};

  \draw[color = blue] (-.5,0) -- (-.5,4); 
 \draw[color = red] (-.5,4) -- (-.5, 6);
 \draw[color = green] (-.5,6) -- (-.5, 10.4);

  \filldraw[color = blue] (-.5,0) circle (2pt);
  \filldraw[color = blue] (-.5,4) circle (2pt);
   \filldraw[color = red] (-.5,6) circle (2pt);

    \draw[color = blue] (.25,-.8) node{ $Q$};
 \draw[color = red] (1.00,-.8) node{ $Z$};
 \draw[color = green] (6,-.8) node{ $1$};

  \draw[color = blue] (0,-.5) -- (.5,-.5); 
 \draw[color = red] (.5,-.5) -- (1.5,-.5);
 \draw[color = green] (1.5,-.5) -- (10.4,-.5);

  \filldraw[color = blue] (0,-.5) circle (2pt);
  \filldraw[color = blue] (.5,-.5) circle (2pt);
   \filldraw[color = red] (1.5,-.5) circle (2pt);

                \end{tikzpicture}

\begin{remark}
    In~\cite[\S4]{serre:sur}  the field  $L$ in the example above is an extension of  $E = \Q_2(\sqrt{5})$.  
\end{remark}

\begin{exercise} 
\label{exc:lowerupperjumps} Set $K = L^Z$.  Show $\Gal(L/K)=\Group(L/K)  =\Group(L/K)_{3/8} = Z$
and 
\( \Group(L/K)_{3/8+} = 1\).
Conclude that $\ell(L/E) = \ell(L/K)$, but $\varphi_{L/E}(3/8)  \neq   \varphi_{L/K}(3/8)$.
\end{exercise}

\begin{exercise} \label{exc:QoverQ2} 
According to the L-functions and modular forms database\footnote{\url{https://www.lmfdb.org/}} (LMFDB) there are, up to isomorphism, three totally ramified quaternionic extensions of $\Q_2$.   Moreover, LMFDB reports that if $L$ denotes any of these extensions of $E = F = \Q_2$, then the jumps for the lower numbering filtration subgroups are $1/8$, $3/8$, and $7/8$.  Show that the graph of $\varphi_{L/E}$ found below is correct.  The label $C_4$ denotes an order four cyclic subgroup of $Q$.

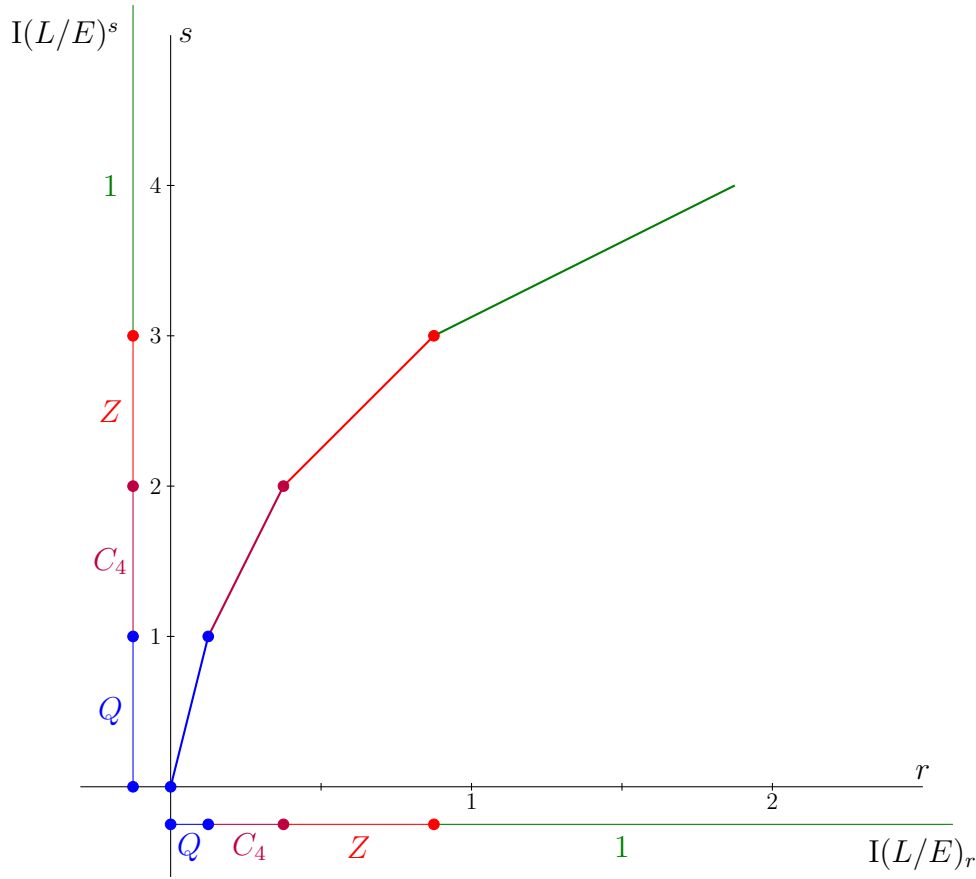
 \begin{tikzpicture}
 \draw (-1.2,0) -- (10,0);
 \draw (0,-1.2) -- (0,10);

\draw (.05,2) -- (-.05,2);
 \draw (.05,4) -- (-.05,4);
\draw (.05,6) -- (-.05,6);
\draw (.05,8) -- (-.05,8);

\draw (2,-.05) -- (2,.05);
\draw (4,-.05) -- (4,.05);
\draw (6,-.05) -- (6,.05);
\draw (8,-.05) -- (8,.05);

  \draw (-.2,2) node { \Scale[.75]{1}};
 \draw (-.2,4) node { \Scale[.75]{2}};
   \draw (-.2,6) node { \Scale[.75]{3}};
 \draw (-.2,8) node { \Scale[.75]{4}};
 \draw (4,-.2) node { \Scale[.75]{1}};
  \draw (8,-.2) node { \Scale[.75]{2}};

  \draw (.2,10) node { \Scale[1]{s}};
     \draw (10,.2) node { \Scale[1]{r}};

   \draw (-1.4,10) node { \Scale[1]{\Group(L/E)^s}};
     \draw (10,-0.9) node { \Scale[1]{\Group(L/E)_r}};

  \draw[color = blue, thick] (0.0,0.0) -- (.5,2.0);
  \draw[color = purple, thick] (.5,2.0) -- (1.5,4);
   \draw[color = red, thick]  (1.5,4) -- (3.5,6);
     \draw[color = green, thick]  (3.5,6) -- (7.5,8);

\filldraw[color = blue] (0,0) circle (2pt);
  \filldraw[color = blue] (.5,2) circle (2pt);
   \filldraw[color = purple] (1.5,4) circle (2pt);
  \filldraw[color = red] (3.5,6) circle (2pt);
  
 \draw[color = blue] (-.8,1) node{ $Q$};
  \draw[color = purple] (-.8,3) node{ $C_4$};
 \draw[color = red] (-.8,5) node{ $Z$};
 \draw[color = green] (-.8,8) node{ $1$};

  \draw[color = blue] (-.5,0) -- (-.5,2); 
    \draw[color =purple] (-.5,2) -- (-.5,4); 
 \draw[color = red] (-.5,4) -- (-.5, 6);
 \draw[color = green] (-.5,6) -- (-.5, 10.4);

  \filldraw[color = blue] (-.5,0) circle (2pt);
  \filldraw[color = blue] (-.5,2) circle (2pt);
   \filldraw[color = purple] (-.5,4) circle (2pt);
    \filldraw[color = red] (-.5,6) circle (2pt);

    \draw[color = blue] (.25,-.8) node{ $Q$};
    \draw[color = purple] (1.05,-.8) node{ $C_4$};
 \draw[color = red] (2.50,-.8) node{ $Z$};
 \draw[color = green] (6,-.8) node{ $1$};

  \draw[color = blue] (0,-.5) -- (.5,-.5); 
   \draw[color = purple] (.5,-.5) -- (1.5,-.5);
 \draw[color = red] (1.5,-.5) -- (3.5,-.5);
 \draw[color = green] (3.5,-.5) -- (10.4,-.5);

  \filldraw[color = blue] (0,-.5) circle (2pt);
  \filldraw[color = blue] (.5,-.5) circle (2pt);
  \filldraw[color = purple] (1.5,-.5) circle (2pt);
   \filldraw[color = red] (3.5,-.5) circle (2pt);

                \end{tikzpicture}

\end{exercise}

\begin{remark} Note that in  Exercise~\ref{exc:QoverQ2} the jumps for the upper numbering filtration of $\Group(L/E)$ are all integers.  Since $Q$ is not abelian, this shows that the converse of the Hasse-Arf theorem~\cite[IV.3 Theorem]{serre:local} is not true.
\end{remark}

\begin{exercise}
According to LMFDB there exist fields $F$, $E$, and $L$ such that
\begin{enumerate}
    \item $L$ is a totally ramified Galois extension of $E$,
    \item $\Gal(L/E) = S_3$, the group of permutations on three distinct letters, and
    \item $\Group(L/E)_{3/e(L/F)}$ is trivial.
\end{enumerate}
Find and graph $\varphi_{L/E}$.
\end{exercise}

\subsubsection{Example: Cyclotomic Extensions of $\Q_p$}   \label{exc:cyclo} This is taken from~\cite[IV.4]{serre:local}.     For  $i, j \in \Z_{\geq 0}$
let
\begin{enumerate}
    \item $\zeta_{p^i}$ be a $p^i$-th primitive root of unity
    \item $L(i) := \Q_p(\zeta_{p^i})$
    \item $G(i) := (\Z/p^i \Z)^\times$  (In~\cite{serre:local} this is denoted $G(p^i)$.)
    \item  $G(i)^j := \{ a \in G(i) \, | \,  a \equiv 1 (p^j) \}$ if $j\leq i$
    and $G(i)^j = 1$ if $i \leq j$.
\end{enumerate}

Fix $n \in \Z_{\geq 1}$.   The Galois group of $L := L(n) $ over $ E = F := L(0) = \Q_p$ is isomorphic to $G(n)$.    For $0 \leq k \leq n$, the quotient $G(n)/G(n)^k$ can be identified with $G(k)$, the Galois group of $L(k)/E$.  Thus, $L^{G(n)^k} = L(k)$.

To compute the ramification subgroups of $\Gal(L/E)$,
the key observation is that $1-\zeta_{p^i}$ is a uniformizer of $L(i)/E$,
or equivalently, that $\val(1-\zeta_{p^i}) = [L(i):E]^{-1}$,
which equals $(p^i - p^{i-1})^{-1}$ if $i\geq1$.
We can use this observation to compute $\depth_{L/E}(\sigma)$.
Suppose $\sigma\zeta_{p^n} = \zeta_{p^n}^a$
with $a\in G(i)^d\setminus G(i)^{d+1}$,
meaning $\zeta_{p^n}^{a-1}$ is a primitive $p^{n-d}$th root of unity.
Then the depth function is given by the formula
\begin{align*}
e(L/E)\cdot\depth_{L/E}(\sigma) &= e(L/E)\cdot \val(\zeta_{p^n}^a - \zeta_{p^n}) - 1
= e(L/E)\cdot \val(\zeta_{p^n}^{a-1} - 1) - 1 \\
&= \frac{p^n-p^{n-1}}{p^{n-d} - p^{n-d-1}} - 1 = p^d - 1.
\end{align*}

For the ramification groups,
define $f \colon \R_{\geq 0} \rightarrow \R_{\geq 0}$
by $f(x) = \lceil \ln_p(1 + e(L/E) \cdot x) \rceil$.
Then for $0 \leq u \leq (p^{n-1} -1)/e(L/E)$,
we have $\Group(L/E)_u = G(n)^{f(u)}$.
In other words,
\begin{equation*}
\Group(L/E)_r  =
\begin{cases}
    G(n)= G(n)^0 \qquad & \text{if $ r = 0$,}\\
    G(n)^m \qquad & \text{if $p^{m-1} -1 < e(L/E) \cdot r \leq p^m -1$,}  \\
    G(n)^n = 1 \qquad & \text{if $p^{n-1} -1 < e(L/E) \cdot r$.}
\end{cases}
\end{equation*}

A direct calculation shows that $\varphi_{L/E}((p^k -1)/{e(L/E)}) = k$
for $k = 0, 1, 2, \ldots, n-1$,
and that, for general $r \in \R_{\geq 0}$, 
\begin{equation*}
    \varphi_{L/E}(r) =
    \begin{cases}
        k + \frac{e(L/E)}{p^{k+1} - p^k} \bigl( r - \frac{p^k-1}{e(L/E)} \bigr) & \qquad 0 \leq k \leq (n-2) \text{ and }\frac{p^{k}-1}{e(L/E)} \leq r < \frac{p^{k+1}-1}{e(L/E)}\\
        (n-1) + \bigl(r - \frac{p^{n-1}-1}{e(L/E)} \bigr) & \qquad \frac{p^{n-1}-1}{e(L/E)} \leq r.
    \end{cases}
\end{equation*}
For upper numbering we have $\Group(L/E)^s = G(n)^{\lceil s\rceil}$.

We graph $\varphi_{L/E}$ below for $p=3$ and $n = 4$.

 
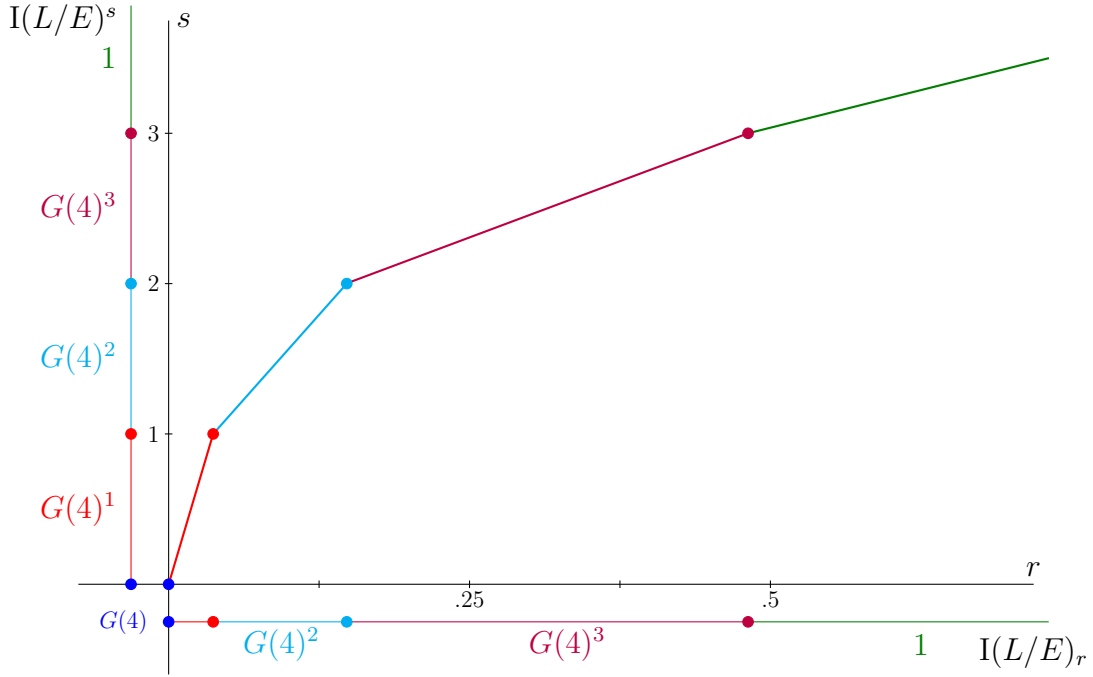
\begin{tikzpicture}
 \draw (-1.2,0) -- (11.5,0);
 \draw (0,-1.2) -- (0,7.5);
 \draw[color = red, thick] (0.0,0.0) -- (.592,2.0);
  \draw[color = cyan, thick] (.592,2.0) -- (2.368,4.0);
   \draw[color = purple, thick]  (2.358,4.0) -- (7.703,6);
    \draw[color = green, thick]  (7.703,6.0) -- (11.703,7);
\draw (.05,2) -- (-.05,2);
\draw (.05,4) -- (-.05,4);
\draw (.05,6) -- (-.05,6);

\draw (2,-.05) -- (2,.05);
\draw (4,-.05) -- (4,.05);
\draw (6,-.05) -- (6,.05);
\draw (8,-.05) -- (8,.05);

 \filldraw[color = blue] (0,0) circle (2pt);
  \filldraw[color = red] (.592,2.0) circle (2pt);
   \filldraw[color = cyan] (2.368,4.0) circle (2pt);
      \filldraw[color = purple] (7.703,6.0) circle (2pt);
  
 \draw (-.2,2) node { \Scale[.75]{1}};
  \draw (-.2,4) node { \Scale[.75]{2}};
 \draw (-.2,6) node { \Scale[.75]{3}};
 \draw (4,-.2) node { \Scale[.75]{.25}};
  \draw (8,-.2) node { \Scale[.75]{.5}};

  \draw (.2,7.5) node { \Scale[1]{s}};
     \draw (11.5,.2) node { \Scale[1]{r}};

   \draw (-1.4,7.5) node { \Scale[1]{\Group(L/E)^s}};
     \draw (11.5,-0.9) node { \Scale[1]{\Group(L/E)_r}};


 \draw[color = blue] (-.6,-.5) node{\Scale[.75]{G(4)}};
 \draw[color = red] (-1.2,1) node{ $G(4)^1$};
 \draw[color = cyan] (-1.2,3) node{ $G(4)^2$};
 \draw[color = purple] (-1.2,5) node{ $G(4)^3$};
 \draw[color = green] (-.8,7) node{ $1$};

 \draw[color = red] (-.5,0) -- (-.5, 2);
  \draw[color = cyan] (-.5,2) -- (-.5, 4);
   \draw[color = purple] (-.5,4) -- (-.5, 6);
 \draw[color = green] (-.5,6) -- (-.5, 7.7);

  \filldraw[color = blue] (-.5,0) circle (2pt);
  \filldraw[color = red] (-.5,2) circle (2pt);
   \filldraw[color = cyan] (-.5,4) circle (2pt);
  \filldraw[color = purple] (-.5,6) circle (2pt);

  \draw[color = red] (0,-.5) -- (.592,-.5); 
 \draw[color = cyan] (.592,-.5) -- (2.368,-.5);
 \draw[color = purple] (2.368,-.5) -- (7.703,-.5);
 \draw[color = green] (7.703,-.5) -- (11.7,-.5);

 \draw[color = cyan] (1.5,-.8) node{ $G(4)^2$};
 \draw[color = purple] (5.3,-.8) node{ $G(4)^3$};
 \draw[color = green] (10,-.8) node{ $1$};
 
  \filldraw[color = blue] (0,-.5) circle (2pt);
  \filldraw[color = red] (.592,-.5) circle (2pt);
   \filldraw[color = cyan] (2.368,-.5) circle (2pt);
   \filldraw[color = purple] (7.703,-.5) circle (2pt);

                \end{tikzpicture}

\subsection{Some observations based on the examples} Suppose $L/E$ is a finite Galois extension. 

\begin{remark}
From the definitions and the pictures we see that the maps $(r \mapsto \Group(L/E)_r)$ and $(s \mapsto \Group(L/E)^s)$ are both right continuous.  That is, for all $b \in \R_{\geq 0}$ there exists an $a<b$ such that the maps are constant on $[a,b] \cap \R_{\geq 0}$.
\end{remark}

\begin{remark}  \label{rem:tamesamephi}
Suppose $E/F$ is a finite extension and $L,K$ are finite Galois extensions of $E$ with $K/E$ tamely ramified.  Then $LK/L$ is also tamely ramified and so $\psi_{KL/L}(s) = s$ and $\psi_{K/E}(s) = s$ for all $s \in \R_{\geq 0}$.  Hence,
\[\psi_{L/E} = \psi_{LK/L} \circ \psi_{L/E} = \psi_{LK/E} = \psi_{LK/K} \circ \psi_{K/E} = \psi_{LK/K}. \]
Thus, by taking $K$ to be the maximal tame subextension $L^t$ in $L$, the study of the function $\psi_{L/E}$ can be reduced to understanding the totally wild case; i.e., $\psi_{L/E} = \psi_{L/L^t}$.  We also have $\varphi_{L/E} = \varphi_{L/L^t}$.
\end{remark}

\begin{example}  Suppose $n \in \Z_{>0}$,
and adopt the language of Example~\ref{exc:cyclo}.  If $L = L(n) = \Q_p(\zeta_{p^n})$, then
$L^t = L^{\Group(L/\Q_p)_{0^+}}= L^{G(n)^1} = L(1) = \Q_p(\zeta_p)$. 
Thus, from  Remark~\ref{rem:tamesamephi} we have $\varphi_{L/\Q_p} = \varphi_{L/\Q_p(\zeta_p)}$.  After accounting for normalizations (see Appendix~\ref{app:comparison}), this agrees with~\cite[Lemma~9 and Proposition~10]{MP19}.
\end{example}

The remainder of this section is based on Exercise~\ref{exc:lowerupperjumps}.

\begin{wrapfigure}{r}{0.25\textwidth} 
    \centering
 \begin{tikzpicture}
     \draw (0.0,6.0) node { \Scale[.75]{L}};
   \draw (0,4.0) node { \Scale[.75]{L^{\Group(L/E)_{\ell+}} = K= L^{\Group(L/E)^{u+}}}};
   \draw (0,2.0) node { \Scale[.75]{E}};
  \draw (-1.2,4.2) -- (0.0,5.8);
  \draw (1.2,3.8) -- (0.0,2.2);
 \draw (-1.0,5.2) node { \Scale[.5]{\Group(L/E)_{\ell+}}};
  \draw (1.6,3.0) node { \Scale[.5]{\Group(L/E)/\Group(L/E)^{u+}}};
                \end{tikzpicture}
\end{wrapfigure}
 Choose $u \in \R_{\geq 0}$ such that $\Group(L/E)^{u:u+}$ is nontrivial.   Note that this immediately implies that $\Group(L/E)$ is nontrivial.
 Let $K = L^{\Group(L/E)^{u+}}$.   We have that $L/K$ is Galois with $\Gal(L/K) = \Group(L/K) = \Group(L/E)^{u+}$.  The extension $K/E$ is Galois with inertia subgroup $\Group(K/E) = \Group(L/E)/\Group(L/K)$.  
 Let $\ell = \psi_{L/E}(u)$.  Then $\Group(L/E)_{\ell} = \Group(L/E)^u$ and $\Group(L/E)_{\ell+} = \Group(L/E)^{u+}$ so $\Group(L/E)_{\ell:\ell+}$ is nontrivial.

Among other things, we will show that the upper jumps for $\Group(K/E)$ are exactly the upper jumps of $\Group(L/E)$ that are less than or equal to  $u$ while the lower jumps  for $\Group(L/K)$ are exactly the 
lower jumps of $\Group(L/E)$ that are greater than  $\ell$.

\begin{lemma}  \label{lem:lowerquotient}
Suppose $r \in \R_{\geq 0}$.
If $r \leq \ell$, then $\Group(L/K)_r = \Group(L/K)_{\ell+}$.   If $r > \ell$, then $\Group(L/K)_{r} = \Group(L/E)_{r}$.
\end{lemma}

\begin{proof}
    Since $\Gal(L/K) = \Group(L/K) = \Group(L/E)_{\ell+}$ and 
    \[\Group(L/K)_r = \Group(L/K) \cap \Group(L/E)_r = \Group(L/E)_{\ell+} \cap \Group(L/E)_r,\]
    we conclude that $\Group(L/K)_{r} = \Group(L/E)_{\ell+}$
    for all $0 \leq r \leq \ell$.

Now assume $r > \ell$. 
From Corollary~\ref{cor:19}
we have the exact sequence
\[ 1 \ra \Group(L/K)_r \ra \Group(L/E)_r \ra \Group(K/E)_{\varphi_{L/K}(r)} \ra 1.\]
Note that $r > \ell$ implies $\Group(L/E)_r \leq \Group(L/E)_{\ell+} = \Group(L/K)$, and so 
\(\Group(K/E)_{\varphi_{L/K}(r)} = \Group(L/E)_r \Group(L/K)/\Group(L/K)\)
is trivial.  The result follows.
\end{proof}

\begin{remark}
    Lemma~\ref{lem:lowerquotient} shows that for all $s > \varphi_{L/E}(\ell)$ we have $\Group(L/E)^{s:s+}$ is nontrivial if and only if $\Group(L/K)^{\psi_{K/E}(s):\psi_{K/E}(s)+}$ is nontrivial.  Moreover, $\psi_{K/E}(s) \leq s$ with equality if and only if $\ell = 0$.  That is, the upper numbering jumps for $L/K$ will be smaller than or equal to their corresponding jumps for  $L/E$ with equality if and only if $\ell = 0$.  In particular, the nonzero upper numbering jumps for $L/L^t$ (here, $L^t=L^{\Group(L/E)^{0+}}$) are exactly the nonzero upper numbering jumps for $L/E$.
\end{remark}

\begin{lemma}  \label{lem:upperonquotient}
Suppose $s \in \R_{\geq 0}$.  We have  $\Group(K/E)^{u+}$ is trivial.
If $0 \leq s \leq u$, then $\Group(K/E)^{s:s+} \cong \Group(L/E)^{s:s+}$.
\end{lemma}

\begin{proof}
  Since  $\Group(L/K) = \Group(L/E)^{u+}$ and
\[  \Group(K/E)^s = \Group(L/E)^s \Group(L/K)/\Group(L/K),\]
it  follows that $\Group(K/E)^s$ is trivial for all $s > u$.  Thus $\Group(K/E)^{u+}$ is trivial.
Now  assume $s \leq u$. From 
Exercise~\ref{exc:alternateformsofexact}
we can derive the commutative diagram
\[
\begin{tikzcd}
1 \arrow[r] & \Group(L/K)^{\psi_{K/E}(s)+} \arrow[r] \arrow[d,hook] & \Group(L/E)^{s+} \arrow[r] \arrow[d,hook] & \Group(K/E)^{s+} \arrow[r] \arrow[d,hook] &1 \\
1 \arrow[r] &  \Group(L/K)^{\psi_{K/E}(s)} \arrow[r] & \Group(L/E)^s \arrow[r]  & \Group(K/E)^s \arrow[r] &1
\end{tikzcd}
\]
where the vertical maps are injective and the horizontal rows are exact.  Since $\psi_{K/E}(s) \leq s \leq u$, we have
\[\Group(L/K)^{\psi_{K/E}(s)} = (\Group(L/E)^{u+})^{\psi_{K/E}(s)}  = \Group(L/E)^{u^+},\]
and
\[\Group(L/K)^{\psi_{K/E}(s)+}= (\Group(L/E)^{u+})^{\psi_{K/E}(s)}  = \Group(L/E)^{u^+} .\]
Hence $\Group(L/K)^{\psi_{K/E}(s)+} =  \Group(L/K)^{\psi_{K/E}(s)} = \Group(L/E)^{u+}$, and the result follows from the snake lemma.
\end{proof}

\subsection{Upper numbering for the absolute inertia subgroup}

Suppose $E/F$ is an extension.
Set 
\[  I(E) =  \{ K \leq \bar{F} \colon \text{$K/E$ is  finite  and Galois}\}. \]
If $K, L \in I(E)$  with  $K \leq L$, then  we have the natural surjection
\[ s_{KL}  \colon \Group(L/E) = \Group(\bar{F}/E) / \Group(\bar{F}/L) \sra \Group(\bar{F}/E) / \Group(\bar{F}/K) = \Group(K/E). \]
This defines an inverse system, and the inverse limit is the inertia subgroup
\[\Group_E = \Group(\bar{F}/E) =  \lim_{\longleftarrow} \Group(K/E) =  \lim_{\longleftarrow} \Group(\bar{F}/E) / \Group(\bar{F}/K).\]
This is, by definition, a profinite group.  Equivalently, $\Group_E$ is a compact totally disconnected topological group~\cite[II.1.4, Theorem~1]{cassels-frohlich}.   A  neighborhood basis of the identity is given by the set $\{ \Group(\bar{F}/K) \, | \, K \in I(E) \}$.   All profinite groups are automatically Hausdorff: the diagonal in $\Group_E \times \Group_E$ is the preimage of the closed set $\{e\}$ under the map $(x,y) \mapsto x y^{-1}$.

Note that the upper numbering filtration is defined so that for all $t \geq 0$,
\[ s_{KL} [\Group(L/E)^t] = \Group(K/E)^t \text{ and }  s_{KL} [\Group(L/E)^{t+}] = \Group(K/E)^{t+}. \]

\begin{definition}
    For $r \geq 0$ we define
\[\Group_E^r = \lim_{\longleftarrow} \Group(K/E)^r \text{ and } \Group_E^{r+} = \lim_{\longleftarrow} \Group(K/E)^{r+}.\]
\end{definition} 

These are closed subgroups of $\Group_E$~\cite[II.1.4, Corollary~2]{cassels-frohlich}.
  Note that $\Group_E = \Group_E^0 = \Group(\bar{F}/E)$ is the inertia subgroup  and $\Group_E^{0+}$ is the wild inertia subgroup of $W_E$.   
  
   Unlike the Moy--Prasad filtration for $p$-adic groups, it is not true that 
 \(\Group_E^{r+} \) is equal to  \( \bigcup_{s > r} \Group_E^s \).   This is because for every $s > r$, there will be an automorphism of $\bar{F}$ that fixes $\bigcap_{t>r} (\bar{F})^{\Group_E^{t}}$ yet does not live in $\Group_E^s$.  Thanks to the fundamental  theorem of Galois theory we know that the absolute Galois group of $\bigcap_{t>r} (\bar{F})^{\Group_E^{t}}$ is the closure of 
 $\bigcup_{t > r} \Group_E^t$, which, by the next lemma, is $\Group_E^{r+}$.

\begin{lemma}
    For $r \in \R_{\geq 0}$, 
    \[ \Group_E^{r+} 
    = \cl \big(  \bigcup_{s > r} \Group_E^s \big).\]
\end{lemma}

\begin{proof}
Since $\Group_E^{r+}$ is closed and $\Group_E^s \subseteq \Group_E^{r+}$ for all $s> r$, it is enough to show  that $\Group_E^{r+}$ is a subset of the closure of $\bigcup_{s>r} \Group_E^s$.

Fix $\sigma \in \Group^{r+}_E$.  Write $\sigma = (\sigma_L)_{L \in I(E)} $. 
Fix\footnote{Sean Cotner points out that we can do this because the number of separable extensions of $E$ is countable. Krasner’s lemma shows that  two separable monic polynomials with nearby coefficients define the same field extension~\cite[Proposition~7.6.1]{milne:ant20}. The space of separable monic polynomials is second-countable (being a product of copies of $E$), so this means there are only countably many separable extensions of a given degree. Thus there are only countably many finite separable extensions, even in positive characteristic.} an increasing sequence 
\[K = K_1 \leq K_2 \leq K_3 \leq \cdots \leq K_n \leq \cdots  \]
with $K_j \in I(E)$ and $\bigcup K_j = \bar{F}$.   To ease notation, define $\sigma_i = \sigma_{K_i}$.   Note that $\sigma = \lim \sigma_i$.

For each $i$ there exists $s_i \in \R_{> r}$ such that $\sigma_i \in \Group(K_i/E)^{s_i}$.   Since for $j \geq i$ we have $\Group(K_j/E)^{s_i} \sra \Group(K_i/E)^{s_i}$, we may assume $s_j \leq s_i$.  Thus, we have a nonincreasing sequence
\[s_1 \geq s_2 \geq \cdots \geq s_n \geq \cdots \]
of nonnegative numbers each of which is greater than $r$.  Since $\Group_E^{s_i}$ surjects onto $\Group(K_i/E)^{s_i}$, we may choose $\tau_i \in \Group_E^{s_i}$ whose image in $\Group(K_i/E)^{s_i}$ is $\sigma_i$.   We then have $\sigma = \lim \tau_i \in \cl (\bigcup_i \Group_E^{s_i}) \subseteq \cl (\bigcup_{s>r} \Group_E^{s})$.
\end{proof}

\begin{lemma}  Let $E/F$ be a finite extension.
\[\bigcap_{s \geq 0} \Group_E^{s} = \bigcap_{s \geq 0} \Group_E^{s+} = 1.\]
\end{lemma}
\begin{proof}
    Suppose $\sigma \in \bigcap_{s \geq 0} \Group_E^{s}$.  Then $\sigma \in \Group_E$, and so we write $\sigma = (\sigma_K)_{K \in I(E)}$.  Fix $K \in I(E)$.  Since $\sigma \in \Group_E^s$ for all $s \geq 0$, we have $\sigma_K \in \Group(K/E)^s$ for all $s$, which means $\sigma_K = 1$.  Thus, $\sigma = 1$.
\end{proof}

\subsection{A different different}  \label{sec:di9fferentdifferent}
Suppose $F \leq E \leq L$ is a tower of finite separable extensions.
The normalized differental exponent of $L/E$ is defined by
    \[
    d(L/E):= \max \{r \in e(L/F)^{-1}\cdot\bbZ \mid
    \Tr_{L/E}[L_{\geq -r}] \subseteq E_{\geq 0} \}.
    \]
In other words, $d(L/E)$ is $e(L/F)^{-1}$ times the usual differental exponent.
The normalized differental exponent  answers the question:
What is the largest $R_L$-lattice $\mathcal{L} \subseteq L$ 
that satisfies  $\Tr_{L/E}[\mathcal{L}] = R_E$? 
We now introduce a measure of the smallest lattice in $L$ with this property, namely
the \emph{compressed differental exponent}
\[
c_{L/E} := d(L/E) - e(E/F)^{-1} + e(L/F)^{-1} \leq d(L/E).
\]

\begin{remark} \label{different-additivity}
It follows from~\cite[VIII.1 Corollary~4]{We95} that for a tower of finite separable extensions $F \leq E \leq K \leq L$ we have 
$$d(L/E) = d(L/K) + d(K/E).$$
After some manipulation this yields
$$c_{L/E} = c_{L/K} + c_{K/E}.$$
\end{remark}

\begin{lemma} \label{thm10}
$c_{L/E}$ is the unique $c\in\bbR$
such that for all $r\in \bbR$,
\[
\Tr_{L/E}[L_{\geq r}] = E_{\geq r+c}.
\]
\end{lemma}

\begin{proof}
Uniqueness is clear: since the Moy--Prasad filtration on $E$ is not constant,
if $E_{\geq r + c} = E_{\geq r + c'}$ for all $r$ then $c=c'$.
Moreover, since the jumps in the filtrations on $L$ and~$E$
lie in $e(L/F)^{-1}\cdot\bbZ$,
we must have that $c_{L/E}\in e(L/F)^{-1}\cdot\bbZ$ as well.
It remains to show that $c_{L/E}$ satisfies the claimed property.

Let $s,s'\in\bbR$.
We first make a few observations
about the effect of $\Tr_{L/E}$ on the fractional ideals of~$L$,
using the $R_E$-linearity and surjectivity of $\Tr_{L/E}\colon L\to E$
together with the relations
$\varpi_E\cdot L_{\geq s} = L_{\geq s+e(E/F)^{-1}}$
and $\varpi_E\cdot E_{\geq s'} = E_{\geq s'+e(E/F)^{-1}}$.
\begin{itemize}
\item
For every $s$ there is some (non-unique) $s'$
such that $\Tr_{L/E}[L_{\geq s} ] = E_{\geq s'}$.

\item
For every $s'$ there is some (non-unique) $s$ 
such that $\Tr_{L/E}[L_{\geq s}] = E_{{\geq s'}}$.

\item
If $\Tr_{L/E}[L_{\geq s}] = E_{{\geq s'}}$
then $\Tr_{L/E}[L_{\geq s + e(E/F)^{-1}}] = E_{\geq s'+e(E/F)^{-1}}$.
\end{itemize}
The problem is to nicely relate such $s$ and~$s'$,
which we may, without loss of generality,  assume lie in $e(L/F)^{-1}\cdot\bbZ$.
Let $d = d(L/E)$ and $c = c_{L/E}$.

Every element of $e(L/F)^{-1}\cdot\bbZ$
can be placed in the interval $[-d,-c]$
after translating by an element of $e(E/F)^{-1}\cdot\bbZ$.
Therefore, using the periodicity property of the filtrations on $L$ and~$E$
under multiplication by $\varpi_E$,
it suffices to show that if $-d\leq r \leq -c$ then $\Tr_{L/E}[L_{\geq r}] = E_{\geq 0}$.
It is enough to check the claim at the endpoints $r=-d$ and $r=-c$.
Certainly $\Tr_{L/E}[L_{\geq -d}] = E_{\geq 0}$ by definition of~$d$,
so that $\Tr_{L/E}[L_{\geq -c}]\subseteq E_{\geq 0}$ also.
If $\Tr_{L/E}[L_{\geq -c}]\subsetneq E_{\geq 0}$,
so that $\Tr_{L/E}[L_{\geq -c}]\subseteq E_{>0}$,
then $\Tr_{L/E}[L_{\geq -c - e(E/F)^{-1}}] \subseteq E_{\geq 0}$,
but this would contradict the maximality in the definition of~$d$
since $c + e(E/F)^{-1} = d + e(L/F)^{-1} > d$.
Hence $\Tr_{L/E}[L_{\geq -c}] = E_{\geq 0}$.
\end{proof}

\begin{corollary}
$c_{L/E} = \min\{r\in e(L/F)^{-1}\cdot\bbZ
\mid \Tr_{L/E}[L_{\geq -r}] = E_{\geq0}\}$. \pushQED{\qed} 
  \qedhere 
    \popQED 
\end{corollary}

We can also use \Cref{thm10}
to explain the change in the depth of an additive character
upon base change.
    Fix a nontrivial continuous additive character $\Lambda=\Lambda_F:F\ra\Cc$.
For any finite separable extension $K/F$, we define $\Lambda_K:=\Lambda_F\circ\Tr_{K/F}:K\ra\Cc$. Since $\Tr_{K/F}:K\ra F$ is surjective, we have that $\Lambda_K$ is also a nontrivial continuous additive character. 
Define
\[
\depth(\Lambda_K):=\max\{ \val(x)\,|\, x\in K,\;\Lambda_K(x)\not=1\}.
\]

\begin{corollary}
$\depth(\Lambda_L) = \depth(\Lambda_E) + c_{L/E}$.  \pushQED{\qed} 
  \qedhere 
    \popQED 
\end{corollary}

To finish, and to set up a connection with the norm map
that we develop in the next section,
we explain the relationship between $c_{L/E}$
and the Herbrand function.

\begin{lemma}[{\cite[IV.2 Proposition~4]{serre:local}}] \label{depth-sum}
$c_{L/E} = \sum_{1 \neq \sigma \in \Group(L/E)} \depth_{L/E}(\sigma)$.\pushQED{\qed} 
  \qedhere 
    \popQED 
\end{lemma}

\begin{corollary}  \label{cor:slargeandc}
  If $s \in \R_{\geq 0}$ satisfies $s \geq \ell(L/E)$, then $\varphi_{L/E}(s) = s + c_{L/E}$.  
\end{corollary}

\begin{proof}
Using the definition of $\varphi_{E/F}$  (see Definition~\ref{def:newphi}),
we see that
\begin{equation*}
    \begin{split}
      \varphi_{L/E}(s) -s &= - s +  \varphi_{L/E}(s) \\
      &= - s + \sum_{\sigma \in \Group(L/E)} \min(\depth_{L/E} (\sigma), s)  \\
      &=  \sum_{\sigma \in \Group(L/E) \setminus \{1\} } \depth_{L/E} (\sigma).  \qedhere
    \end{split}
\end{equation*}
\end{proof}

\begin{remark}
    We conclude from Corollary~\ref{cor:slargeandc} that $c_{L/E} \geq 0$ with equality if and only if $L/E$ is tamely ramified.
\end{remark}

\subsection{Some results about norm maps}

If $L/E$ is a finite Galois extension, then for all $s \in \R$ we have
\[ \Tr_{L/E}[L_{\geq s}] = E_{\geq (s + c_{L/E})} .\]
Since this is true for all $s$, we have a surjection, also denoted $\Tr_{L/E}$,
\[  \Tr_{L/E} \colon L_{= s} \sra E_{=(s + c_{L/E})}. \]

There are similar, but more subtle, results for the norm map $\Nm_{L/F} \colon L^\times \ra F^\times$.  For example, thanks to~\cite[V.2~Proposition 2 and V.6~Proposition 8]{serre:local}, for all $s \in \R_{>0}$ we have
\[ \Nm_{L/E}[L^{\times}_{\geq s}] =\Nm_{L^u/E}[\Nm_{L/L^u}[L^{\times}_{\geq s}]] \subseteq \Nm_{L^u/E}[L^{u \times}_{\geq \varphi_{L/L^u}(s)} ] = E^{\times}_{\geq \varphi_{L/L^u}(s)}.\]
Here $L^u/E$ denotes the maximal unramified subextension of $E$ in $L$. 
Since $\varphi_{L^u/E}(t) = t$ for all $t \geq 0$, we conclude that $\Nm_{L/E}[L^\times_{\geq s}] \subseteq E^\times_{\geq \varphi_{L/E}(s)}$.  Since this is true for all $s \in \R_{>0}$,  we have a homomorphism, also denoted $\Nm_{L/E}$,
\[  \Nm_{L/E} \colon L^\times_{= s} \longrightarrow E^\times_{=\varphi_{L/E}(s)}. \]

\begin{lemma} \label{lem:normrange}
   Suppose $s \geq 0$.
   \begin{enumerate}
       \item \label{it:norm1} The homomorphism 
    \[  \Nm_{L/E} \colon L^\times_{\geq s} \longrightarrow E^\times_{\geq \varphi_{L/E}(s)}\]
    is surjective if and only if either $L$ is unramified and $s \geq 0$ or $\Group(L/E)$ is nontrivial and $s > \ell(L/E)$.
    \item \label{it:norm2}  Suppose $s > \ell(L/E)$. If $x \in L_{\geq s}$, then $\Nm_{L/E}(1+x) = 1 +\Tr_{L/E}(x)$ modulo $E^\times_{> \varphi_{L/E}(s)}$.  Since   we know from Corollary~\ref{cor:slargeandc} that $s > \ell(L/E)$ implies $\varphi_{L/E}(s) = s + c_{L/E}$,  we have the following commutative diagram
\[
\begin{tikzcd}
L_{=s}\arrow[rr, "\Tr_{L/E}"]\arrow[d, "\rotatebox{90}{$\sim$}"]& &E_{=(s+c_{L/E})}\arrow[d, "\rotatebox{90}{$\sim$}"]\\
L^{\times}_{=s}\arrow[rr, "\Nm_{L/E}"]& &E^{\times}_{=(s + c_{L/E})}
\end{tikzcd}
\]
 in which the vertical maps are the $x \mapsto 1 + x$ isomorphisms.
   \end{enumerate}
\end{lemma}

\begin{proof}
    Suppose first that $L/E$ is unramified.  In this case   the statements in the Lemma follow from~\cite[V.2 Propositions 1 and 3]{serre:local} and their proofs.  
    
    Now suppose  that $L/E$  is totally ramified and cyclic of prime order.  Since
    \[ \Nm_{L/E} \colon L^\times_{= \ell(L/E)} \ra E^\times_{=\varphi_{L/E}(\ell(L/E))} \]
    is neither injective nor surjective~\cite[V.3 Proposition 5(iii)]{serre:local},  it follows that if $s \leq \ell(L/E)$, then  $\Nm_{L/E}[L^\times_{\geq s}] \subsetneq E^\times_{\geq \varphi_{L/E}(s)}$.  On the other hand, for $s > \ell(L/E)$, 
    the proof of~\cite[V.3 Proposition 5(iv)]{serre:local}  shows that Claim~\ref{it:norm2} of the Lemma holds.  It follows that  for all $s > \ell(L/E)$ we have $\Nm_{L/E}[L^\times_{\geq s}] = E^\times_{\geq \varphi_{L/E}(s)}$.

   Suppose now that $L/E$ is a totally ramified Galois extension.  We shall proceed by induction on $|\Group(L/E)|$.  If $\Group(L/E)$ is trivial, then $E = L$ and there is nothing to prove.  If $\Group(L/E)$ is not trivial, then  since $\Group(L/E)$ is solvable there is a subextension $K/E$ of $L/E$ such that $\Group(K/E) = \Group(L/E)/\Group(L/K)$ is cyclic of prime order.  By induction we know that for all $r > 0$ 
   \begin{enumerate}
   \item 
    \(  \Nm_{L/K} [L^\times_{\geq r}] \subseteq  K^\times_{\geq \varphi_{L/K}(r)}\) with equality if and only if $r > \ell(L/K)$ and 
    \item \label{it:norm2a} if  $r > \ell(L/K)$ and  $x \in L_{\geq r}$, then $\Nm_{L/K}(1+x) = 1 +\Tr_{L/K}(x)$ modulo $K^\times_{> \varphi_{L/K}(r)}$. 
    \end{enumerate}
   Since $K/E$ is cyclic of prime order, we know that for all $t > 0$
    \begin{enumerate}
   \item 
    \(  \Nm_{K/E} [K^\times_{\geq t}] \subseteq  E^\times_{\geq \varphi_{K/E}(t)}\) with equality if and only if $t > \ell(K/E)$ and 
    \item \label{it:norm2b} if  $t > \ell(K/E)$ and  $y \in K_{\geq t}$, then $\Nm_{K/E}(1+y) = 1 +\Tr_{K/E}(y)$ modulo $E^\times_{> \varphi_{K/E}(t)}$. 
    \end{enumerate}
For all $s > 0$ we have 
\begin{equation*}
    \begin{split}
    \Nm_{L/E}[L^\times_{\geq s}] &= \Nm_{K/E}[\Nm_{L/K}[L^\times_{\geq s}]] \\ 
    &\subseteq \Nm_{K/E}[K^\times_{ \geq \varphi_{L/K}(s)}] \\
    &\subseteq E^\times_{\geq \varphi_{K/E}(\varphi_{L/K}(s))} = E^\times_{\geq \varphi_{L/E}(s)} 
     \end{split}
\end{equation*}
 with equality if and only if both $s> \ell(L/K)$ and $\varphi_{L/K}(s) > \ell(K/E)$.  But, from Corollary~\ref{cor:exact2} we know that  $s> \ell(L/K)$ and $\varphi_{L/K}(s) > \ell(K/E)$ if and only if $s > \ell(L/E)$.  Thus, Claim~\ref{it:norm1} holds for totally ramified Galois extensions. Finally, if $r > \ell(L/E)$, then for all $x \in L_{\geq r}$ we have
 \begin{equation*}
     \begin{split}
         \Nm_{L/E}(1 + x) & = \Nm_{K/E} ( \Nm_{L/K}(1 + x)) \\
         &= \Nm_{K/E} (1 + \Tr_{L/K}(x) + z) \qquad \hphantom{+ w}\text{     for some $z \in K_{>\varphi_{L/K}(r)}$}\\
          &= 1 + \Tr_{K/E}(\Tr_{L/K}(x) + z) + w  \qquad \text{ for some $w \in E_{>\varphi_{K/E}(\varphi_{L/K}(r))}$}\\
          &\equiv 1 + \Tr_{L/E}(x)  \mod{E^\times_{>\varphi_{L/E}(r)}},
     \end{split}
 \end{equation*}
and so Claim~\ref{it:norm2} holds for totally ramified Galois extensions.

 Finally, suppose $L/E$ is any finite Galois extension.  The case when $L/E$ is unramified is handled above, so we may assume $\Group(L/E)$ is nontrivial.    We let $K = L^{\Group(L/E)}$ be the maximal unramified extension of $E$ inside $L$.  Since $L/K$ is totally ramified with Galois group $\Group(L/K) = \Group(L/E)$ and $\varphi_{L/E} = \varphi_{L/K}$,  from our results above we have that for all $s > 0$
 \[ \Nm_{L/E}[L^\times_{ \geq s}] = \Nm_{K/E}[\Nm_{L/K}[L^\times_{\geq s}]] \subseteq \Nm_{K/E}[K^\times_{\geq \varphi_{L/K}(s)}] = E^\times_{\geq \varphi_{L/K}(s)} = E^\times_{\geq \varphi_{L/E}(s)}\]
 with equality if and only if $s > \ell(L/K) = \ell(L/E)$.   For $r > \ell(L/E)$ our results above similarly yield $\Nm_{L/E}(1+x) \equiv 1 + \Tr_{L/E}(x) \mod{E^\times_{>\varphi_{L/E}(r)}}$.
\end{proof}

With a bit more work, we can also describe
the interaction at intermediate depth
between the norm map and congruence subgroups.

\begin{theorem}[{\cite[V.6]{serre:local}}]  \label{thm11}
Let $L/K$ be a finite totally ramified Galois extension
and let $r\geq0$. 
If $K^\times_{=\varphi_{L/K}(r)}$ is nontrivial, 
then there is an exact sequence
\[
\begin{tikzcd}
1 \rar &
\Group(L/K)^{r:r+} \rar{\theta} &
L^\times_{=r} \rar{\Nm_{L/K}} &
K^\times_{=\varphi_{L/K}(r)}.
\end{tikzcd}
\]
\end{theorem}

We omit a proof, but it uses a similar reduction to cyclic Galois extensions
that we saw in the proof of \Cref{lem:normrange}.
Since  $L/K$ is totally ramified, we have that if both $L^\times_{=t}$ and $K^\times_{=s}$
are nonzero for some $t, s \in \R_{>0}$, then they have the same cardinality.  Thus,
 \Cref{thm11} implies that  the homomorphism
$\Nm_{L/K} \colon L^\times_{=r} \to K^\times_{=\varphi_{L/K}(r)}$ is, outside a finite number of exceptional depths,  usually
an isomorphism when the target is nonzero.
The set of exceptional depths is a subset
of the set of jumps in the upper numbering filtration on $\Group(L/K)$, and it is
 in general a proper subset:
$\Group(L/K)^{r:r+}\neq1$ does not imply that $K^\times_{=\varphi_{L/K}(r)} \neq 1$.

\begin{example}[Wildly ramified norm-one tori]
Let $E = F$ and suppose $L/E$ is a totally ramified separable quadratic extension.
Let $T$ be the $F$-torus representing
the norm-one elements $L^1\defeq\ker(\Nm_{L/F})$.
Using Corollary~\ref{cor:slargeandc} to rewrite  Example~\ref{ex:quadratic} we have
\[
\varphi_{L/F}(r) = \begin{cases}
2r & \tn{if $0\leq r\leq c_{L/F}$} \\
r + c_{L/F} & \tn{if $c_{L/F} < r$.}
\end{cases}
\]
From local Artin reciprocity~\cite[Theorem~1.1]{milne:lcft} there is an exact sequence
\[
1 \longrightarrow L^1 \longrightarrow L^\times
\xlongrightarrow{\Nm} F^\times \longrightarrow \Group(L/F)
\longrightarrow 1.
\]
Heuristically, \Cref{thm11} implies that
the mass of $L^1$ concentrated at shallow depths is missing.
More precisely: $L^1 = L^1_{\geq c_{L/F}}$,
or in other words, $L^1_{=r} = 1$ if $r < c_{L/F}$.
Moreover, the parahoric subgroup of $L^1$ is $L^1_{>c_{L/F}}$
and the relative component group is $L^1_{=c_{L/F}}\cong\bbZ/2\bbZ$,
as predicted by the Kottwitz homomorphism \cite[Cor.~11.1.6]{KP23}.
\Cref{fig:mass} illustrates the effect
of the norm map on the depth, and the resulting
congruence filtration of $L^1$.
Specifically, suppose $S$ is one of $T(F)$, $\Gm(L)$ or $\Gm(F)$.
At depth~$r$,
we draw \begin{tikzpicture}\node[empty] (A) at (0,0) {};\end{tikzpicture}
if $S_{=r}=1$, 
we draw \begin{tikzpicture}\node[half] (A) at (0,0) {};\end{tikzpicture}
if $S_{=r}\simeq\bbZ/2\bbZ$, and 
we draw \begin{tikzpicture}\node[full] (A) at (0,0) {};\end{tikzpicture}
if either $r>0$ and $S_{=r} \simeq \bbF_q$ or $r = 0$ and $S_{=0} \simeq \bbF_q^\times$.
We use a similar notational scheme for $\Group(L/F)^{r:r+}$.

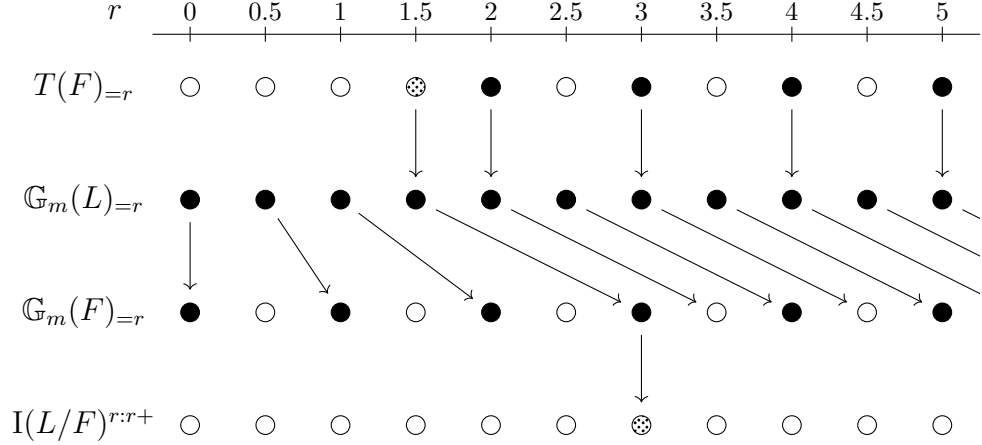
\begin{figure}[h]
\centering
\begin{tikzpicture}
\clip (-2.5,-2) rectangle (10.5,5);

\foreach \a in {-1,0,...,14}{
	\coordinate (A\a) at (\a,3);
	\coordinate (B\a) at (\a,1.5);
	\coordinate (C\a) at (\a,0);
	\coordinate (D\a) at (\a,-1.5);
}

\node (la-1) at (A-1) {$T(F)_{=r}\qquad$};
\node (lb-1) at (B-1) {$\Gm(L)_{=r}\qquad$};
\node (lc-1) at (C-1) {$\Gm(F)_{=r}\qquad$};
\node (ld-1) at (D-1) {$\Group(L/F)^{r:r+} \qquad$};

\node (R) at (-1,4) {$r$};
\draw (-0.5,3.7) edge (10.5,3.7);
\foreach \a in {0,1,...,10} \draw (\a,3.6) -- (\a,3.8);

\foreach[evaluate={\b=int(\a/2)}] \a in {0,2,...,10} \node (dp\a) at (\a,4) {\footnotesize$\b$};
\foreach[evaluate={\b=\a/2}] \a in {1,3,...,9} \node (dp\a) at (\a,4) {\footnotesize$\b$};

\foreach \a in {0,1,2,5,7,9} \node[empty] (a\a) at (A\a) {};
\node[half] (a3) at (A3) {};
\foreach \a in {4,6,8,10} \node[full] (a\a) at (A\a) {};
\foreach \a in {0,1,...,12} \node[full] (b\a) at (B\a) {};
\foreach \a in {0,2,...,12} \node[full] (c\a) at (C\a) {};
\foreach \a in {1,3,...,11,13} \node[empty] (c\a) at (C\a) {};
\foreach \a in {0,1,2,3,4,5,7,8,9,10} \node[empty] (d\a) at (D\a) {};
\node[half] (d6) at (D6) {};

\foreach \a in {3,4,6,8,10} \draw[->] (a\a) -- (b\a);

\draw[->] (b0) -- (c0);
\draw[->] (b1) -- (c2);
\draw[->] (b2) -- (c4);
\draw[->] (b3) -- (c6);
\draw[->] (b4) -- (c7);
\draw[->] (b5) -- (c8);
\draw[->] (b6) -- (c9);
\draw[->] (b7) -- (c10);
\draw[->] (b8) -- (c11);
\draw[->] (b9) -- (c12);
\draw[->] (b10) -- (c13);

\draw[->] (c6) -- (d6);
\end{tikzpicture}
\caption{The effect of the norm map on the congruence filtration,
for $L/F$ wild quadratic with $c_{L/F}=3/2$.}
\label{fig:mass}
\end{figure}
\end{example}

\begin{remark}[Norm map for non-Galois extensions]
Using local class field theory,
we can show that if $L/K$ is a finite separable,
not necessarily Galois, extension
then $\Nm_{L/K}(L^\times) = \Nm_{L_0/K}(L_0^\times)$
where $L_0/K$ is the maximal Galois extension of~$K$ contained in~$L$,
or in other words, the intersection of all Galois conjugates of~$L$.  (To see this, note that for $L_a$, the maximal abelian subextension of $L$, we have $\Nm_{L/K}(L^\times) = \Nm_{L_a/K}(L_a^\times)  \leq \Nm_{L_0/K}(L_0^\times) \leq \Nm_{L/K}(L^\times)$~\cite[III, Theorem~3.5]{milne:lcft}.)  
In particular, if $L/K$ is ``totally non-Galois'' in the sense that $L_0=K$
then $\Nm_{L/K}$ is surjective.
There are many such extensions, for example,
$K(\sqrt[e]{\varpi})/K$ if $e$ is prime to $|k_K| - 1$.
It would be interesting to generalize \Cref{thm11} to such extensions.
\end{remark}
 
\subsection{Upper numbering subgroups and extensions}

In this section $E$ is an extension of $F$ and $L/E$ is a finite Galois extension.

\begin{lemma} \label{lem:18} Suppose $r \in \R_{\geq 0}$.   We have
\[\Group_E^r \cap \Group_L^0 = \Group_L^{\psi_{L/E}(r)} \text{ and } \Group_E^{r+} \cap \Group_L^0 = \Group_L^{\psi_{L/E}(r)+}.\]
Moreover, if $\Group_E^r \leq \Group^0_L$, then  for all $s \geq r$ we have
\begin{enumerate} 
    \item $\Group_L^{\psi_{L/E}(s)} = \Group_E^s$  and
    \item $\psi_{L/E}(s) - s = \psi_{L/E}(r) - r$.
\end{enumerate} 
\end{lemma}

\begin{proof}
Suppose $M/E$ is a finite Galois extension such that $L \leq M$.

For the first statement, it will be enough to show that $\Group(M/L)^{\psi_{L/E}(s)} = \Group(M/E)^s \cap \Group(M/L)$ for all $s \geq r$.   Suppose $s \geq r$. We have
\begin{equation*}
    \begin{split}
       \Group(M/E)^s \cap \Group(M/L) 
        &= \Group(M/E)_{\psi_{M/E}(s)} \cap \Group(M/L) = \Group(M/L)_{\psi_{M/E}(s)} \\
        &=  \Group(M/L)^{\varphi_{M/L} (\psi_{M/E}(s))} = \Group(M/L)^{\psi_{L/E}(s)}.
    \end{split}
\end{equation*}

We now assume $\Group_E^r \leq \Group_L^0$ and turn to the remaining statements.   From the computations above we have
\[  \Group(M/E)^s  =  \Group(M/E)^s \cap \Group(M/L)  = \Group(M/L)^{\psi_{L/E}(s)}.\]
For the second statement, note that 
for all $t \geq \psi_{M/E}(r)$ we have
\[ \Group(M/E)_t =  \Group(M/E)^{\varphi_{M/E}(t)} =  \Group(M/L)^{\psi_{L/E}(\varphi_{M/E}(t))} =  \Group(M/L)^{\varphi_{M/L}(t)} 
=  \Group(M/L)_t.\]
Thus
\begin{equation*}
    \begin{split}
        \psi_{L/E}(s) - s &= \varphi_{M/L}(\psi_{M/E} (s)) -  \varphi_{M/E}(\psi_{M/E} (s)) \\
        &= \varphi_{M/L}(\psi_{M/E} (r)) + \int_{\psi_{M/E}(r)}^ {\psi_{M/E} (s)}  \big| \Group(M/L)_t \big| \, dt \\
        &  \hspace{2em}  - \Bigr( \varphi_{M/E}(\psi_{M/E} (r))  + \int_{\psi_{M/E} (r)}^ {\psi_{M/E} (s)}  \big| \Group(M/E)_t \big| \, dt \Bigl) \\  
        &= \varphi_{M/L}(\psi_{M/E} (r)) - \varphi_{M/E}(\psi_{M/E} (r)) \\
        &= \psi_{L/E}(r) - r. \qedhere
    \end{split}
\end{equation*}
\end{proof}

\begin{remark}  From Exercise~\ref{exc:classsical_sublemma} the  first claim of Lemma~\ref{lem:18} is consistent with the classical version of the same result with $F = E$ (see, for example, \cite[Proposition 1.4]{BH19}, \cite[Lemma~1]{MP19}, or \cite[Lemma~2.2]{AP22}; in~\cite{AP22} they call this result the \emph{comparison lemma}). 
\end{remark}

\begin{definition}  Recall that $E$ is an extension of $F$ and $L/E$ is a finite Galois extension.  We have defined
    \( \ell(L/E)  = \inf \{r \in \R_{\geq 0} \colon \Group(L/E)_r = 1  \}\),
 and we now define
    \[  u(L/E)  := \inf \{t \in \R_{\geq 0} \colon \Group(L/E)^t = 1  \}. \]
\end{definition}

\begin{example} \label{ex:exampofuandl}
In Example~\ref{exc:cyclo} we have
\[u(\Q_p(\zeta_{p^n})/\Q_p) = n-1 \, \text{   and    } \, \ell(\Q_p(\zeta_{p^n})/\Q_p) = \frac{p^{n-1} - 1}{(p-1)p^{n-1}}.\]
\end{example}

\begin{remark}  \label{rem:sometruthsaboutuandl}  Here are some observations about $\ell(L/E)$ and $u(L/E)$.
\begin{enumerate}
    
\item \label{it:one} We have \[ u(L/E) = \varphi_{L/E}(\ell(L/E)) \text{ and } \psi_{L/E}( u(L/E)) = \ell(L/E),\]
 and thus from Corollary~\ref{cor:slargeandc} we conclude
 that $c_{L/E} = u(L/E) - \ell(L/E)$.

\item
From~\cite[IV.2 Exc. 3]{serre:local} we have 
 \[ \ell(L/E) \leq \frac{ e(p)}{(p-1)\cdot e(L/F)}\]
where $e(p)$ denotes the absolute ramification index of $p$. 
\end{enumerate}
\end{remark}

\begin{lemma}\label{lem:TFAE}
Suppose $s \in \R_{\geq 0}$.  The following statements are equivalent.
\begin{enumerate}

     \item   \label{item:five} \( \psi_{L/E}(s) \geq \ell(L/E)\)  or \( s \geq u(L/E)\).
   
 \item \label{item:six}   $ c_{L/E} = s - \psi_{L/E}(s) $.
    
    \item  \label{item:two}  \(\Group(L/E)_{\psi_{L/E}(s)+} = 1 \).
    
\item  \label{item:three} \( \Norm_{L/E} [L^\times_{>\psi_{L/E}(s)}] = E^\times_{>s} \).

\item    \label{item:four} $\Group_L^{\psi_{L/E}(s)+} = \Group_E^{s+}$. 

   \item \label{item:one} \(\Group_E^{s+} \leq \Group^0_L    \).

\end{enumerate} 
\end{lemma}

\begin{proof} 

(\ref{item:five}) $\Leftrightarrow$ (\ref{item:six}):
Since $s \geq u(L/E)$ if and only if $\psi_{L/E}(s) \geq \ell(L/E)$, it is enough to show that $\psi_{L/E}(s) \geq \ell(L/E)$ if and only if $c_{L/E} = s - \psi_{L/E}(s)$.

Thanks to Corollary~\ref{cor:slargeandc} we conclude  that if $\psi_{L/E}(s) \geq \ell(L/E)$, then   $c_{L/E} = \varphi_{L/E}(\psi_{L/E}(s))  - \psi_{L/E}(s) = s - \psi_{L/E}(s)$.

Suppose now that $c_{L/E} = s - \psi_{L/E}(s)$.  Recall from Remark~\ref{rem:sometruthsaboutuandl}(\ref{it:one}) that $c_{L/E} = u(L/E) - \ell(L/E)$.
Note that  $t \mapsto (t - \psi_{L/E}(t))$ is a non-decreasing function on $\R_{\geq 0}$ that is strictly increasing on $[0, u(L/E))$ and constant  on $[u(L/E), \infty)$.  Thus,  if $s - \psi_{L/E}(s) = c_{L/E} = u(L/E) - \ell(L/E)$, then $s \geq u(L/E)$.

(\ref{item:five}) $\Leftrightarrow$ (\ref{item:two}):
Since  \( \psi_{L/E}(s) \geq \ell(L/E)\)  if and only if \( s \geq u(L/E)\), it is enough to show (\ref{item:two}) holds if and only if  \( \psi_{L/E}(s) \geq \ell(L/E)\).  This follows  from the definition of $\ell(L/E)$.

(\ref{item:two}) $\Leftrightarrow$ (\ref{item:three}): 
Since $\Group(L/E)_{\psi_{L/E}(s)+} =1$ if and only if $\psi_{L/E}(s) \geq \ell(L/E)$, this follows from  Lemma~\ref{lem:normrange}~(\ref{it:norm1}).

(\ref{item:two}) $\Leftrightarrow$ (\ref{item:four}):
Suppose $t \geq 0$ and $M/E$ is a finite Galois extension with $L \leq M$.   We need to show \(\Group(M/L)^{\psi_{L/E}(t)} = \Group(M/E)^t \) if and only if $\Group(L/E)_{\psi_{L/E}(t)}$ is trivial.
From 
Exercise~\ref{exc:alternateformsofexact} we have that 
\[ 1 \ra \Group(M/L)^{\psi_{L/E}(t)} \ra \Group(M/E)^t  \ra \Group(L/E)_{\psi_{L/E}(t)} \ra 1\]
is exact.   From this it follows that $\Group(L/E)_{\psi_{L/E}(t)}$ is trivial if and only if 
\[\Group(M/L)^{\psi_{L/E}(t)} = \Group(M/E)^t.\]

(\ref{item:four}) $\Leftrightarrow$ (\ref{item:one}):  
From Lemma~\ref{lem:18} we know that $\Group_L^{\psi_{L/E}(s)+}  = \Group_E^{s+} \cap \Group_L^0$.  Since $\Group_L^0 \leq \Group_E^0$, this means $\Group_L^{\psi_{L/E}(s)+} \leq \Group_E^{s+}$.    Thus, it will be enough to show 
$\Group_E^{s+} \leq \Group_L^{\psi_{L/E}(s)+}$ if and only if $\Group_E^{s+} \leq \Group_L^0$.

If   
 $\Group_E^{s+} \leq \Group_L^{\psi_{L/E}(s)+}$, then  $\Group_E^{s+} \leq \Group_L^{0}$.  On the other hand, if $\Group_E^{s+} \leq \Group_L^{0}$, then $\Group_E^{s+} = \Group_E^{s+} \cap \Group_L^{0} = \Group_L^{\psi_{L/E}(s)+}$ from Lemma~\ref{lem:18}, so $\Group_E^{s+} \leq \Group_L^{\psi_{L/E}(s)+}$.
\end{proof}

\begin{remark}
If $L/E$ is tamely ramified (so $\ell(L/E) = u(L/E) = 0$), then  $\varphi_{L/E}(t) = \psi_{L/E}(t) =  t$ for all $t \geq 0$, 
so $\Group_E^{v} = \Group_L^v$ for all $v > 0$.  If $L/E$ is not tamely ramified, then we have  
$\Group_E^{\ell(L/E)+} = \Group_L^{\psi_{L/E}(\ell(L/E))+}$, but $\Group_E^{\ell(L/E)}$ is not a subset of $\Group_L^0$.  This was observed in~\cite[1.9 Corollary~2]{BH19}.
\end{remark}

\section{Depth preservation, or lack thereof}

\subsection{Cohomology}

Suppose $W$ is a topological group with a neighborhood basis of the identity consisting of open normal subgroups.  A discrete topological group $M$ on which $W$ acts will be called a \emph{smooth $W$-group} provided that $\Stab_W(m)$ is open for all $m \in M$.  Note that $M$ is not assumed to be abelian.

Suppose $M$ is a smooth $W$-group and 
$I \leq W$ is normal in $W$.  We endow $W/I$ with the quotient topology.   Let $M^{I}$ denote the  $I$-fixed points in $M$; it is a smooth $W/I$-group.
A function $c \colon W/I \ra M^{I}$ is called a \emph{(continuous) $1$-cocycle} or  \emph{(continuous) crossed homomorphism} provided that
\begin{itemize}
   \item  $c(\tau \mu) = c(\tau) \cdot \tau c(\mu)$ for all $\tau, \mu \in W/I$, and
\item there exists an open $J \leq W/I$ such that $c(\tau \sigma) = c(\tau)$ for all $\tau \in W/I$,  $\sigma \in J$.
\end{itemize}
Let $Z^1(W/I,M^{I})$ denote the set of continuous $1$-cocycles from $W/I$ to $M^{I}$.
Two cocycles $c_1, c_2 \in Z^1(W/I,M^{I}) $ are said to be \emph{cohomologous} provided that there exists  $m \in M^{I}$ such that $c_1({\tau}) = m^{-1} c_2({\tau}) \tau(m)$ for all $\tau \in W/I$.   This is an equivalence relation, and we let $H^1(W/I,M^{I})$ denote the set of equivalence classes in $Z^1(W/I,M^{I})$.

\begin{lemma} 
Suppose $I_1 \leq I_2$ are normal subgroups of $W$.  The natural map
\begin{equation} \label{equ:injection} \tag{$\heartsuit$} H^1(W/I_2,M^{I_2}) \rightarrow H^1(W/I_1,M^{I_1})
\end{equation}
is an injection.
\end{lemma}

\begin{proof}
For $\tau \in W/I_1$, let $\bar{\tau}$ denote its image in $W/I_2$.
Suppose $c_1, c_2 \in Z^1(W/I_2,M^{I_2})$ have the same image in $H^1(W/I_1,M^{I_1})$.  Then there exists $m \in M^{I_1}$ such that $c_1(\bar{\tau}) = m^{-1} c_2(\bar{\tau}) \tau(m)$ for all $\tau \in W/I_1$.  Since $\bar{\tau} = 1$ for all $\tau \in I_2$ and $c_1(1) = c_2(1) = 1$, we conclude that $m = \rho(m)$ for all $\rho \in I_2$.   That is, $m \in M^{I_2}$.  Hence, $c_1$ and $c_2$ are cohomologous.  
\end{proof}

The result below is ~\cite[Lemma~2.9]{AP22}.

\begin{lemma}
Suppose $E/F$ is a finite Galois extension, $M$ is a smooth $W_E$-group, and $r \in \R_{\geq 0}$. 
We have
\[H^1(W_E,M) = \lim_{\longrightarrow} H^1(W_E/\Group_E^{r},M^{\Group_E^{r}} )\]
and
\[H^1(W_E,M) = \lim_{\longrightarrow}  H^1(W_E/\Group_E^{r+},M^{\Group_E^{r+}} ).\]
\end{lemma}

 \begin{proof}  
  We prove the first statement. The proof of the second follows by replacing $r$ with $r+$ in what follows.

Setting $I_1 = 1$ in (\ref{equ:injection}), we see that the right-hand side injects into the left-hand side. For the other direction, choose $c \in Z^1(W_E,M)$. We need to show that $c \in Z^1(W_E/\Group_E^r,M^{\Group_E^r})$ for some $r$.  Since $c$ is continuous, there exists an $r\in \R_{\geq 0}$ such that $c(\tau \sigma) = c(\tau) $  for all $\tau \in W_E$ and for all $\sigma \in \Group_E^r$. Thus,  for all $\sigma \in \Group_E^r$ we have
     \[ c(1) =  c(1 \cdot \sigma) = c(1) c( \sigma).\]
    We conclude that $c(\sigma) = c(1) = 1$ for all $\sigma \in \Group_E^r$.  Fix  $\tau \in W_E$ and $\sigma \in \Group_E^r$.  Since $ \tau^{-1} \sigma \tau \in \Group_E^r$, we have 
    \[c(\tau) = c(\tau (\tau^{-1} \sigma \tau)) = c(\sigma \tau) = c(\sigma) \sigma c(\tau) = \sigma c(\tau).\]
Since this is true for all $\tau \in W_E$ and $\sigma \in \Group_E^r$, we conclude that $c \in Z^1(W_E/\Group_E^r,M^{\Group_E^r})$.
 \end{proof}

Suppose $W'$ is a subgroup of $W$ and $M$ is a smooth $W'$-group. 
Let $\Ind_{W'}^{W}M$ denote the group of  maps $f \colon W \rightarrow M$ such that
\begin{enumerate}
    \item $f(hw) = h \cdot f(w)$ for all $h \in W'$ and $w \in W$ and
    \item there is an open subgroup $I$ of $W$ such that $f(wi) = f(w)$ for all $i \in I$ and $w \in W$.
\end{enumerate} 
This is a smooth $W$-group under the right-regular action: $(w \cdot f)(x) = f(xw)$ for all $x, w \in W$.   In this context, Shapiro's lemma~\cite[Proposition 1.29]{BS64} says:
\begin{lemma}  \label{lem:shapirolem}  Suppose $M$ is a smooth $W'$-group.
 There is an isomorphism
\[ \Sh \colon H^1(W, \Ind_{W'}^{W}M) \xlongrightarrow{\sim} H^1(W',M)\]
induced by the map $c \mapsto (h \mapsto c(h)(1))$.  Here $c\in Z^1(W,\Ind_{W'}^{W} M)$ and $h \in W'$. \pushQED{\qed} 
  \qedhere 
    \popQED 
\end{lemma}

\subsection{On the depth of a Langlands parameter}
We present the definition of depth given in~\cite[Section~2]{AP22}.  
Let $\Fr$ denote a Frobenius element in $W_F$.

Suppose $G$ is a connected reductive $F$-group, and we let $G^\vee$ denote its dual.  An $L$-parameter for $G$ is a homomorphism 
\[ \varphi \colon  \SL_2  \times W_F \ra {}^LG :=   G^{\vee}  \rtimes W_F\]
such that
\begin{enumerate}  
    \item  the natural map $\pr_2 \circ \res_{W_F} \varphi $ from $W_F$ to itself is the identity
    \item  $\res_{\Group_F^0} \varphi \colon \Group_F^0 \ra {^{L}}G$ is continuous and $\varphi(\Fr)$ is semisimple   
    \item  $\res_{\SL_2}\varphi$ is a homomorphism of algebraic groups over $\C$.
\end{enumerate}

\begin{definition} Suppose $\varphi$ is an $L$-parameter for $G$.  
\begin{enumerate}
    \item 
We define $c_{\varphi} \in Z^1(W_F,G^\vee)$ by $\varphi(1,\sigma) = (c_{\varphi}(\sigma), \sigma)$ and let $[c_{\varphi}]$ denote the image of $c_{\varphi}$ in $H^1(W_F,G^\vee)$.
\item For $g \in G^\vee$ we define the $L$-parameter ${}^g\varphi$ for $G$ by ${}^g\varphi(\tau) = g \varphi(\tau) g^{-1}$ for $\tau \in W_F$.
\item  We define the depth, $\depth_F(\varphi)$, of $\varphi$ to be
\[\depth_F(\varphi) := \inf \{ r \in \R_{\geq 0} \, | \, [c_{\varphi}] \in H^1(W_F/\Group_F^{r+}, (G^\vee)^{\Group_F^{r+}}).\]
\end{enumerate}
\end{definition}

\begin{remark} If $G$ is $F$-split, then $c_{\varphi} = \res_{W_F} \varphi$.
\end{remark}

\begin{lemma}  Every element in the $G^\vee$-conjugacy class of an $L$-parameter $\varphi$ of $G$ has the same depth as $\varphi$.
\end{lemma}

\begin{proof}  Suppose $\varphi$ is an $L$-parameter for $G$ and  $g \in G^\vee$.  A computation shows that $c_{{}^g\varphi} (\sigma) =  g \cdot c_{\varphi}(\sigma) \cdot \sigma(g)^{-1}$ for all $\sigma \in W_F$.  Thus
     $[c_{\varphi}] = [c_{{}^g\varphi}]$, and   the result follows.
\end{proof}

\begin{lemma}
    If $G$ splits over a tamely ramified extension and $\varphi$ is an $L$-parameter for $G$, then
    \[\depth_F(\varphi) = \inf \{ r \in \R_{\geq 0} \, | \, c_{\varphi} [\Group_F^{r+}] \subseteq G^\vee \text{ is trivial}\}.\]
\end{lemma}

\begin{proof}
    Since $G$ splits over a tamely ramified extension, the wild inertia subgroup $\Group_F^{0+}$ acts trivially on $G^\vee$.  Thus, for $r \in \R_{\geq 0}$ we have
   \(  [c_{\varphi}] \in H^1(W_F/\Group_F^{r+}, (G^\vee)^{\Group_F^{r+}})\)
   if and only if
    \( [c_{\varphi}] \in H^1(W_F/\Group_F^{r+}, G^\vee)\)
    if and only if
    \( c_{\varphi} [\Group_F^{r+}] \subseteq G^\vee \text{ is trivial}\).
\end{proof}

\subsection{Depth and restriction of scalars on the Galois side}
In this section we present a result of Aubert and Plymen~\cite[Theorems~2.16 and~2.17]{AP22}, which is a generalization of the same result for tori by Mishra and Pattanayak~\cite[Lemma~4 and Corollary~5]{MP19}.

We begin with a result that Aubert and Plymen~\cite[Lemma~2.8]{AP22} call the \emph{submodule lemma}~\cite[Lemma 3]{MP19}.  
\begin{lemma} \label{lem:submodule} \pushQED{\qed} 
Suppose $W$ is a topological group and $I,W' \leq W$ with $I$ normal in $W$.  Then the map $f \mapsto (gI \mapsto f(g))$ is an isomorphism
\[\big( \Ind_{W'}^W M \big)^I \cong   \Ind_{W'/(W' \cap I)}^{W/I}(M^{(W' \cap I)})\]
for all smooth  $W'$-groups $M$.
    \qedhere 
    \popQED  
\end{lemma}

\begin{exercise}
    Prove Lemma~\ref{lem:submodule}.
\end{exercise}

Suppose $E/F$ is a finite Galois extension and $H$ is a connected reductive $E$-group.  Let $\Res_{E/F}$ denote the Weil restriction of scalars of $H$ (see~\cite[Section 2.i]{Mi17}) for an introduction to restriction of scalars for algebraic groups).

\begin{lemma} \label{lem:64}  For all $r \geq 0$ there is a canonical isomorphism
    \[ H^1 \bigr(W_F/\Group_F^{r+}, [(\Res_{E/F} H)^\vee]^{\Group_F^{r+}} \bigr) \cong H^1 \bigr(W_E/\Group_E^{\psi_{E/F}(r)+}, (H^\vee)^{\Group_E^{\psi_{E/F}(r)+}} \bigr). \]
\end{lemma}
\begin{proof}
Note that as smooth $W_F/\Group_F^{r+}$ groups, we have
\begin{equation*}
    \begin{split}
        \big[(\Res_{E/F} H)^\vee\big]^{\Group_F^{r+}} & \cong  \big[\Ind_{W_E}^{W_F} H^\vee\big] ^{\Group_F^{r+}} \\
& \qquad \text{(thanks to Lemma~\ref{lem:submodule}}    )     \\
& \cong    \Ind_{W_E/(W_E \cap \Group_F^{r+})}^{W_F/\Group_F^{r+}} (H^\vee)^{\Group_F^{r+} \cap W_E}\\
& \qquad \text{(thanks to Lemma~\ref{lem:18}}    )  \\
& =   \Ind_{W_E/ \Group_E^{\psi_{E/F}(r)+}}^{W_F/\Group_F^{r+}} (H^\vee)^{\Group_E^{\psi_{E/F}(r)+}}.\\
    \end{split}
\end{equation*}
 The Shapiro isomorphism (Lemma~\ref{lem:shapirolem}) then tells us
\begin{equation*}
    \begin{split}
        H^1   &\bigl( W_F/\Group_F^{r+}, [(\Res_{E/F} H)^\vee]^{\Group_F^{r+}} \bigr) \\ &\cong H^1 \bigl(W_F/\Group_F^{r+},  \Ind_{W_E/ \Group_E^{\psi_{E/F}(r)+}}^{W_F/\Group_F^{r+}} (H^\vee)^{\Group_E^{\psi_{E/F}(r)+}} \bigr) \\
        &\cong H^1 \bigl(W_E/\Group_E^{\psi_{E/F}(r)+}, (H^\vee)^{\Group_E^{\psi_{E/F}(r)+}}\bigr). \qedhere
         \end{split}
\end{equation*}
\end{proof}

\begin{corollary}
    Suppose $H$ is a connected reductive $E$ group and $\varphi$ is an $L$-parameter for $H$.  Let $\Sh(\varphi)$ denote the $L$-parameter for $\Res_{E/F} (H)$ that arises via the Shapiro isomorphism.  Then
    \pushQED{\qed} 
    \[ \depth_F(\Sh(\varphi) ) = \psi_{E/F}(\depth_E(\varphi)). \qedhere \]
    \popQED
\end{corollary}

\subsection{The Local Langlands Correspondence and depth for induced tori}

Let $E/F$ be a finite Galois extension.

Set $T = \Res_{E/F} \Gm$, the Weil restriction of scalars of $\Gm$.  Then $X^*(T) := \Hom(T,\Gm)$ is (canonically) a free $\Z$-module with basis $\Gal(E/F) \cong W_F/W_E$, and $T^\vee := X^*(T) \otimes_{\Z} \C^\times$ is (canonically)  $\Ind_{W_E}^{W_F} \C^\times = \oplus_{W_F/W_E} \C^\times$.
Thus, from Shapiro's Lemma~\ref{lem:shapirolem} we conclude that 
\[H^1(W_F,T^\vee) \cong H^1(W_F,\Ind_{W_E}^{W_F} \C^\times) \cong  H^1(W_E,\C^\times) = \Hom(W_E,\C^\times).\]
On the other hand $T(E) \cong E^\times$, and so the local class field theory isomorphism $\theta_E \colon   E^\times \xra{\sim} W_E^{\ab}$~\cite[V.1 Exc. 11]{neu99}, gives us
\[ \Hom(T(F),\C^\times) = \Hom(E^\times,\C^\times) \cong \Hom(W_E,\C^\times).\]
The composition of these maps yields the Local Langlands Correspondence (LLC) for induced tori:
\[ \LLC \colon \Hom(T(F),\C^\times) \xra{\sim} H^1(W_F,T^\vee).\]

\begin{definition}  The depth, $\depth_{T(F)}(\chi)$,  of $\chi \in \Irr(T(F))$ is defined to be
\[\depth_{T(F)}(\chi) := \min \{ r \in \Q_{\geq 0} \, | \, \chi[T(F)_{r+}] = 1\}. \]
\end{definition}

The following result is~\cite[Corollary 7]{MP19}.

\begin{lemma}  If $\chi \in \Irr(T(F))$, then
\[ \depth_{T(F)}(\chi) =  \psi_{E/F}(\depth_F( \LLC(\chi))).\]
\end{lemma}

\begin{proof}  For all $s \in \R_{\geq 0}$ we have
    \begin{equation*}
        \begin{split}
            \Hom(T(F)/T(F)_{>s},\C^\times) &=  \Hom(E^\times/E^\times_{>s},\C^\times) \\
            & (\text{\cite[VI.4.1, Theorem 1]{cassels-frohlich} and }\\
            &\qquad \text{\cite[XIII.4, Proposition~12]{serre:local}}) \\
            &= \Hom(W_E/\Group_E^{s+},\C^\times) \\
             &= H^1(W_E/\Group_E^{s+},\C^\times) \\
            & \text{  (from Lemma~\ref{lem:64})} \\
            &= H^1(W_F/\Group_F^{\varphi_{E/F}(s)+},[\Ind^{W_E}_{W_F} \C^\times]^{\Group_F^{\varphi_{E/F}(s)+}}) \\
             &= H^1(W_F/\Group_F^{\varphi_{E/F}(s)+},[T^\vee]^{\Group_F^{\varphi_{E/F}(s)+}}) 
        \end{split}
    \end{equation*}
    Since this holds for all $s$, the lemma follows.
\end{proof}

 If $\chi$ has depth zero or $E/F$ is tamely ramified, then we conclude that 
\[ \depth_{T(F)}(\chi) =  \depth_F( \LLC(\chi)).\]  However, if $\chi$ has positive depth and $\Group(E/F)_{>0}$ is not trivial, then since $\psi_{E/F}(x) < x$ for all $x > 0$, we will have 
\[ \depth_{T(F)}(\chi) =  \psi_{E/F}(\depth_F( \LLC(\chi)) ) < \depth_F(\LLC(\chi))\]
and for all $\chi$ of depth greater than or equal to  $\ell(E/F)$ we will have 
\[\depth_F( \LLC(\chi)) = \depth_{T(F)}(\chi) + c_{E/F}.\]

\begin{example}
    Suppose $F = \Q_p$, $E = \Q_p(\zeta_{p^n})$ with $n$ large, and $S_n$ is the $F$-torus $\Gm \times R_{E/F}(\Gm)$.  If $\chi \in \Irr(S_n(F))$, then $\chi = \chi_1 \otimes \chi_2$ with $\chi_1 \in \Irr(\Gm(F))$ and $\chi_2 \in \Irr(R_{E/F}(\Gm)(F))$.  Suppose $\chi_1$ has depth $r$ with $r \in \Z_{>1}$ and $\chi_2$ has depth $s$ with $s \cdot e(E/F) \in \Z_{>1}$.  Note that if $u(E/F) \leq s < r < s + c_{E/F}$, then $\depth_{S_n(F)}(\chi) = r$ while $\depth_F(\LLC(\chi)) = s + c_{E/F}$. 
    Since $c_{E/F} > n-2$ (see Example~\ref{ex:exampofuandl}), there is a very large number of characters on $S_n(F)$ for which the depths of the characters on the  $p$-adic and Galois sides are ``independent'' of each other.  
\end{example}

\begin{remark}
    A new definition of depth on the Galois side was recently proposed in~\cite{Mis25}.  This new definition is designed to preserve depth.
\end{remark}

\appendix

\section{Comparison with classical definitions} \label{app:comparison}\noindent Suppose $F \leq E \leq L$ with $L/E$ finite Galois. 

For $s \in \R_{\geq 0}$ let $\tilde{\Group}(L/E)_s$ denote the classical lower numbering filtration subgroup of $\Gal(L/E)$.   Recall that in the classical
setting the valuation  is normalized so that the valuation of $L^\times$ is $\Z$.   This means that if $\tilde{\Group}(L/E)_{r:r+} \neq 1$, then $r \in \Z$. We have
\[{\Group}(L/E)_s = \tilde{\Group}(L/E)_{e(L/F) \cdot s}.\]
 We define \(\tilde{\varphi}_{L/E} \colon \R_{\geq 0} \rightarrow \R_{\geq 0}\) by
\[\tilde{\varphi}_{L/E}(x)  := \int_0^x [\tilde{\Group}(L/E)_0: \tilde{\Group}(L/E)_t]^{-1} \, dt\]
and its inverse \(\tilde{\psi}_{L/E} \colon \R_{\geq 0} \rightarrow \R_{\geq 0}\) is the classical Hasse-Herbrand function.

\begin{lemma}
    For all $x \geq 0$ we have
    \[\varphi_{L/E}(x) = \frac{1}{e(E/F)} \tilde{\varphi}_{L/E} (e(L/F) \cdot x)\]
    and for all $y \geq 0$ we have
       \[\psi_{L/E}(y) = \frac{1}{e(L/F)} \tilde{\psi}_{L/E} (e(E/F) \cdot y).\]
\end{lemma}

\begin{proof}
 Since $\psi_{L/E}$ is the inverse of $\varphi_{L/E}$ and  $\tilde{\psi}_{L/E}$ is the inverse of $\tilde{\varphi}_{L/E}$, it is enough to verify the first equality.
 
   Suppose $x \in \R_{\geq 0}$.   We have
    \begin{equation*}
        \begin{split}
            \varphi_{L/E}(x) &= \int_0^x \left| \Group(L/E)_s \right| \, ds = e(L/E) \int_0^x \frac{\left| \Group(L/E)_s \right|}{\left| \Group(L/E)_0 \right|} \, ds\\
            &= e(L/E) \int_0^x \left[ \Group(L/E)_0 : \Group(L/E)_s \right]^{-1} \, ds\\
             &= e(L/E) \int_0^x \left[ \tilde{\Group}(L/E)_0 : \tilde{\Group}(L/E)_{e(L/F) \cdot s} \right]^{-1} \, ds\\
              &= \frac{e(L/E)}{e(L/F)} \int_0^{e(L/F) \cdot x} \left[ \tilde{\Group}(L/E)_0 : \tilde{\Group}(L/E)_{t} \right]^{-1} \, dt\\
              &= \frac{1}{e(E/F)} \tilde{\varphi}_{L/E}(e(L/F) \cdot x). \qedhere
        \end{split}
    \end{equation*}
\end{proof}

For $t \geq 0$, the upper numbering filtration  subgroup in the classical setting is defined by
\[ \tilde{\Group}(L/E)^t := \tilde{\Group}(L/E)_{\tilde{\psi}(t)}.\]

\begin{lemma} \label{lem:uppernumberingnorm} For $t \geq 0$ we have
\[ \Group(L/E)^t = \tilde{\Group}(L/E)^{e(E/F) \cdot t}. \]
\end{lemma}

\begin{remark}  If $E$ is an unramified extension of $F$, then the classical indexing of the  upper numbering filtration subgroups agrees with the indexing presented in this note.
\end{remark}

\begin{proof}
    Suppose $t \in \R_{\geq 0}$.  We have
    \begin{equation*}
        \begin{split}
            \Group(L/E)^t &= \Group(L/E)_{\psi_{L/E}(t)} = \tilde{\Group}(L/E)_{e(L/F) \cdot \psi_{L/E}(t)}\\
            &= \tilde{\Group}(L/E)^{\tilde{\varphi}_{L/E}(e(L/F) \cdot \psi_{L/E}(t))} \\
             &= \tilde{\Group}(L/E)^{e(E/F) \cdot {\varphi}_{L/E}(\psi_{L/E}(t))} \\
             &= \tilde{\Group}(L/E)^{e(E/F) \cdot t }. \qedhere
        \end{split}
    \end{equation*}
\end{proof}

\begin{exercise}  \label{exc:classsical_sublemma}  Suppose $M$ is a finite Galois extensions of $E$ with $L \leq M$.  Use Lemma~\ref{lem:18} to 
    show
  \[    \tilde{\Group}(M/F)^r \cap \tilde{\Group}(M/L) = \tilde{\Group}(M/E)^{\tilde{\psi}_{L/F} ( r)}. \]
\end{exercise}

\section{Non-Galois extensions} \label{sec:nongalois}
\noindent Here we briefly explain how most of the basic definitions
of ramification theory can be extended to arbitrary finite
separable extensions, not necessarily Galois.
We leave it as an exercise to the reader
to fully work out the generalization in this setting.
The appendix to \cite{deligne84} has many steps in this direction as well.

Let $E/F$ be a finite separable extension
and let $\bar{E}$ be a separable closure of~$E$.
Let $L/E$ be a finite separable extension
which, in this \namecref{sec:nongalois},
we do not necessarily assume is contained in~$\bar{E}$.
As before, let $\val$ be the unique valuation
$\val\colon\bar E\to\bbQ\cup\{\infty\}$ such that $\val(F^\times)=\bbZ$.

When $L/E$ is not Galois, we can simply define the Herbrand functions
by reducing to the Galois case:
choose a finite Galois extension $\tilde{L}/E$ containing~$L$ and set
\[
\varphi_{L/E} := \varphi_{\tilde{L}/E} \circ \varphi^{-1}_{\tilde{L}/L}.
\]
Alternatively, and perhaps more conceptually,
it is possible to extend our basic definitions
to non-Galois extensions directly,
rather than first passing through the Galois case, as follows.

Much of Galois theory can be extended to separable,
non-Galois extensions by replacing the Galois group $\Gal(L/E)$
with the $\Gal(\bar E/E)$-set
$\Hom_{E\tn{-alg}}(L,\bar E)$ of $E$-algebra embeddings $L\to\bar E$.
This is the perspective of Grothendieck's Galois theory \cite{sga1} --
see \cite{lenstra} for a textbook treatment --
whose simplest manifestation is the basic fact that
\[
\big|{\Hom_{E\tn{-alg}}}(L,\bar E)\big| = [L:E].
\]
Moreover, among the discrete $\Gal(\bar E/E)$-sets,
the transitive sets are precisely those isomorphic to $\Hom_{E\tn{-alg}}(L,\bar E)$
for some finite separable extension~$L/E$.
This relationship connects ideas from Galois theory
to the category of discrete $\Gal(\bar E/E)$-sets,
a connection that we will extend to ramification theory.

Restricting this analogy to $\Group(L/E)$ rather than $\Gal(L/E)$ is possible,
but slightly awkward.
To do so, we fix an embedding $\iota\colon k_L\to k_{\bar E}$
and define $\Hom_{E\tn{-alg},\iota}(L,\bar E)$
to be the subset of $\Hom_{E\tn{-alg}}(L,\bar E)$
of maps inducing $\iota$ on residue fields.  
Writing $\Group_E$ for the inertia subgroup of~$\Gal(\bar E/E)$,
the natural action of $\Group_E$ on $\bar E$ makes
$\Hom_{E\tn{-alg},\iota}(L,\bar E)$ into a discrete transitive $\Group_E$-set.
The action of $\Aut(L/E)$ on $L$
does not induce an action of $\Aut(L/E)$ on $\Hom_{E\tn{-alg},\iota}(L,\bar E)$,
simply because if $\sigma\in\Aut(L/E)$
induces the automorphism $\bar\sigma\in\Gal(k_L/k_E)$,
then $\sigma$ will map $\Hom_{E\tn{-alg},\iota}(L,\bar E)$
to $\Hom_{E\tn{-alg},\iota\circ\bar\sigma}(L,\bar E)$.
However, the inertia subgroup
\[
\Aut_0(L/E) \defeq \ker\bigl(\Aut(L/E)\to\Gal(k_L/k_E)\bigr)
\]
does act on $\Hom_{E\tn{-alg},\iota\circ\bar\sigma}(L,\bar E)$,
since it preserves $\iota$,
and in fact, $\Aut_0(L/E)$ is the group of
$\Group_E$-equivariant automorphisms of
$\Hom_{E\tn{-alg},\iota\circ\bar\sigma}(L,\bar E)$.

Given $\lambda,\mu\in\Hom_{E\tn{-alg},\iota}(L,\bar E)$, define
\[
\depth_{L/E}(\lambda,\mu) \defeq
\val\left(\frac{\lambda(\varpi_L) - \mu(\varpi_L)}{\lambda(\varpi_L)}\right).
\]
The non-inertial version of this definition is due to Weil
\cite[page~148]{We95}.
To compare this notation with its previous meaning,
suppose $L/E$ is Galois with $L\subseteq\bar E$.
If we write $\lambda^{-1}\mu$ for the unique element of
$\Group(L/E)$ such that $\lambda\circ(\lambda^{-1}\mu) = \mu$, then
$\depth_{L/E}(\lambda,\mu) = \depth_{L/E}(\lambda^{-1}\mu)$.

The function $\depth_{L/E}$ satisfies the following properties for all $\lambda, \mu, \tau \in \Hom_{E\tn{-alg},\iota}(L,\bar E)$,
generalizing \Cref{depth-properties}:
\begin{enumerate}
\setcounter{enumi}{-1}
\item $\depth_{L/E}(\lambda,\mu) = \infty$ if and only if $\lambda=\mu$.
\item $\depth_{L/E}(\lambda,\mu) = \depth_{L/E}(\mu,\lambda)$.
\item $\depth_{L/E}(\lambda,\mu)
\geq \min\bigl(\depth_{L/E}(\lambda,\tau),\depth_{L/E}(\tau,\mu)\bigr)$ 
\end{enumerate}
It follows that  the formula
\begin{equation} \label{metric}
d(\lambda,\mu) \defeq p^{-\depth_{L/E}(\lambda,\mu)}
\end{equation}
defines an ultra metric on $\Hom_{E\tn{-alg}}(L,\bar E)$.

This construction extends to any discrete $\Group_E$-set~$X$ as follows. 
The function $\depth_{L/E}$ is $\Group_E\times\Aut_0(L/E)$-invariant
in the sense that for every $\sigma\in\Group_E$ and $\alpha\in\Aut_0(L/E)$,
\begin{equation} \label{depth-equivariance}
\depth_{L/E}(\lambda,\mu) = \depth_{L/E}(\sigma\circ\lambda\circ\alpha,
\sigma\circ\mu\circ\alpha).
\end{equation}
When $\Group_E$ acts transitively on $X$, there is some $L/E$
for which there exists a $\Group_E$-equivariant isomorphism
$i\colon X\simeq\Hom_{E\tn{-alg},\iota}(L,\bar{E})$.
We can use this isomorphism to define a depth function on $X$ by
\[
\depth_X(x,y) \defeq \depth_{L/E}(i(x),i(y)).
\]
Using \eqref{depth-equivariance},
we see that this definition is independent of the choice of~$i$.
For a general discrete $X$, we decompose $X$ uniquely as a disjoint union
$X = \bigsqcup_{i\in I} X_i$ of homogeneous $\Group_E$-sets~$X_i$,
choose a negative constant such as $-1$,
and define a depth function on~$X$ as follows,
where $x_i\in X_i$ and $y_j\in X_j$:
\[
\depth_X(x_i,y_j) \defeq \begin{cases}
\depth_{X_i}(x_i,y_j) & \tn{if }i=j \\
-1 & \tn{if }i\neq j.
\end{cases}
\]
Mimicking \eqref{metric}, for $x,y \in X$ we set $d(x,y) \defeq p^{-\depth_X(x,y)}$.
In this way we have an $\Group_E\times\Aut_{\Group_E}(X)$-invariant
metric on any discrete $\Group_E$-set~$X$.

One advantage of extending ramification theory to non-Galois extensions
is that it gives a new interpretation of the depth function as a distribution on~$\Group_E$, see~\cite[page~152]{We95}.

\begin{lemma}[Weil]
The following assignment on basic open sets of $\Group_E$
extends to a distribution on~$\Group_E$:
\[
H(1_{\sigma\Group_L}) = \begin{cases}
    -c_{L/E} & \tn{if $\sigma\in\Group_L$} \\
    \depth_{\Group_E/\Group_L}(\sigma\Group_L,\Group_L) & \tn{if $\sigma\notin\Group_L$.}
\end{cases}
\]
Here $L$ ranges over finite subextensions of $\bar E/E$ and $\sigma \Group_L \in \Group_E/\Group_L$.
\end{lemma}

\noindent In other words, to spell out the definitions above in this particular case,
the tautological embedding $\tn{taut}\colon L \hookrightarrow \bar E$
induces an identification between
$\Group_E/\Group_L$ and $\Hom_{E\tn{-alg},\iota}(L,\bar E)$,
and if $\sigma \Group_L$ corresponds to $\mu\colon L\to\bar E$,
then $\depth_{\Group_E/\Group_L}(\sigma \Group_L,\Group_L)
= \depth_{L/E}(\mu,\tn{taut})$.  

\begin{proof}
Let $L/K/E$ be a tower of finite separable extensions contained in~$\bar E$
and suppose the coset $\tau\Group_K \in \Group_E/\Group_K$ decomposes into $\Group_L$-cosets as
$\tau\Group_K = \bigsqcup_{i\in I} \sigma_i\Group_L$.
We need to show that
\[
H(1_{\tau\Group_K}) = \sum_{i\in I} H(1_{\sigma_i\Group_L}).
\]
If $\tau\notin\Group_K$, then this claim amounts to
a generalization of \Cref{lem:firstformula},
\[
\depth_{\Group_E/\Group_K}(\tau\Group_K, \Group_K)
= \sum_{i\in I} \depth_{\Group_E/\Group_L}(\sigma_i\Group_L,\Group_L),
\]
which is \cite[(11), page~152]{We95} up to our renormalization.
If $\tau\in\Group_K$, then this claim amounts to the equation
\[
-c_{K/E} = -c_{L/E} +
\sum_{i\in I,\,\sigma_i\notin\Group_L}
\depth_{\Group_E/\Group_L}(\sigma_i\Group_L,\Group_L),
\]
or, by the additivity of $c$ (see \Cref{different-additivity}),
the equation
\[
\sum_{i\in I,\,\sigma_i\notin\Group_L}
\depth_{\Group_E/\Group_L}(\sigma_i\Group_L,\Group_L) = c_{L/K}.
\]
This generalization of \Cref{depth-sum} follows from \cite[(9), page~149]{We95}.
\end{proof}

\subsection*{Acknowledgments} 
We thank the attendees of the Harish-Chandra seminar at the University of Michigan for their patience and questions during our talks on this material.

\def\cfgrv#1{\ifmmode\setbox7\hbox{$\accent"5E#1$}\else
  \setbox7\hbox{\accent"5E#1}\penalty 10000\relax\fi\raise 1\ht7
  \hbox{\lower1.05ex\hbox to 1\wd7{\hss\accent"12\hss}}\penalty 10000
  \hskip-1\wd7\penalty 10000\box7}
\providecommand{\bysame}{\leavevmode\hbox to3em{\hrulefill}\thinspace}
\providecommand{\MR}{\relax\ifhmode\unskip\space\fi MR }
\providecommand{\MRhref}[2]{%
  \href{http://www.ams.org/mathscinet-getitem?mr=#1}{#2}
}
\providecommand{\href}[2]{#2}


\begin{thebibliography}{ABPS16}


\bibitem[AP22]{AP22}
A.-M.~Aubert and R.~Plymen,
\emph{Comparison of the depths on both sides of the local
              {L}anglands correspondence for {W}eil-restricted groups},
 {with an appendix by Jessica Fintzen},
{J. Number Theory},
\textbf{233}(2022), pp.~24--58.
   



\bibitem[BS64]{BS64}
A.~Borel and J.-P. Serre,
\emph{Th\'eor\`emes de finitude en cohomologie galoisienne},
{Comment. Math. Helv.},
\textbf{39}({1964}), pp.~{111--164}.
    


\bibitem[BH06]{BH:GLtwo}
C.~Bushnell and G.~Henniart,
\emph{The local {L}anglands conjecture for {$\rm GL(2)$}},
Grundlehren der mathematischen Wissenschaften 
{335},
 {Springer-Verlag, Berlin},
{2006}.

  
\bibitem[BH19]{BH19}
\bysame, 
\emph{Local {L}anglands correspondence and ramification for
              {C}arayol representations},
{Compos. Math.},
\textbf{155}({2019}), no. 10, pp.~{1959--2038}.



\bibitem[CF67]{cassels-frohlich}
J.~W.~S.~Cassels and A.~Fr\"ohlich (eds.),
\emph{Algebraic number theory},
{Proceedings of an instructional conference organized by the
              {L}ondon {M}athematical {S}ociety (a {NATO} {A}dvanced {S}tudy
              {I}nstitute) with the support of the {I}nternational
              {M}athematical {U}nion},
 {Academic Press, London; Thompson Book Co., Inc., Washington,
              DC},
 {1967}.


\bibitem[Del84]{deligne84}
P.~Deligne,
\emph{Les corps locaux de caract\'eristique {$p$}, limites de corps
              locaux de caract\'eristique {$0$}},
{Representations of reductive groups over a local field}, {Travaux en Cours},
{Hermann, Paris},
{1984},
      pp.~{119--157}.
   

 
 \bibitem[SGA1]{sga1}
A.~Grothendieck and M.~Raynaud, 
 \emph{Rev\^etements \'etales et groupe fondamental ({SGA} 1)},
Documents Math\'ematiques (Paris),
{3},  [\emph{S\'eminaire de g\'eom\'etrie alg\'ebrique du Bois Marie
              1960--61}, 
              updated and annotated reprint of the 1971 original],
{Soci\'et\'e{} Math\'ematique de France, Paris},
{2003}.
     


\bibitem[Hil97]{Hil97}
D.~Hilbert,
\emph{Die Theorie der algebraischen Zahlkörper},
 Jahresber. Dtsch. Math., Ver. 4(1897),  pp.~175--546.



 \bibitem[KP23]{KP23}
 T.~Kaletha and G.~Prasad, \emph{Bruhat-{T}its theory---a new approach},
   New Mathematical Monographs, vol.~44, Cambridge University Press, Cambridge,
   2023. \MR{4520154}

\bibitem[Len06]{lenstra}
H.W.~Lenstra, 
\emph{Galois theory for schemes, third edition}, 2006, \url{https://websites.math.leidenuniv.nl/algebra/GSchemes.pdf}


\bibitem[Mil17]{Mi17}
J.~Milne, 
\emph{Algebraic groups. The theory of group schemes of finite type
over a field},
Cambridge Studies in Advanced Mathematics, {170}, Cambridge
University Press, Cambridge, 2017.


\bibitem[Mil20a]{milne:ant20}
\bysame,
\emph{Algebraic Number Theory (v3.08)},
{2020},
\url{ www.jmilne.org/math/}.

\bibitem[Mil20b]{milne:lcft}
\bysame,
\emph{Class Field Theory (v4.03)},
{2020},
\url{www.jmilne.org/math/}.


\bibitem[Mis25]{Mis25}
M.~Mishra,
\emph{Depth preservation and close-field transfer in the Local Langlands Correspondence},
\url{
https://doi.org/10.48550/arXiv.2509.04997
}.


\bibitem[MP19]{MP19}
M.~Mishra and B.~Pattanayak,
\emph{A note on depth preservation},
{J. Ramanujan Math. Soc.},
\textbf{34}(2019), no. 4, pp.~{393--400}.
  



\bibitem[Neu99]{neu99}
J.~Neukirch,
\emph{Algebraic number theory},
Grundlehren der mathematischen Wissenschaften, {322},
{Springer-Verlag, Berlin},
 {1999}.









\bibitem[Ser79]{serre:local}
J.-P. Serre,
\emph{Local fields},
{Graduate Texts in Mathematics}, vol. {67},
{Translated from the French by M.J.~Greenberg},
{Springer-Verlag, New York-Berlin}, {1979}.


\bibitem[Ser60]{serre:sur}
\bysame,
\emph{Sur la rationalit\'e{} des repr\'esentations d'{A}rtin},
{Ann. of Math. (2)},
\textbf{72}(1960), pp.
 {405--420}.
   
  

 \bibitem[Wei95]{We95}
 A.~Weil,
 \emph{Basic number theory},
 {Springer-Verlag, Berlin},
 {1995}.



\end{thebibliography}
\end{document}